\documentclass[11pt,oneside]{amsart}
\allowdisplaybreaks

\usepackage{graphicx, amsmath, amssymb, amscd, amsthm, euscript, psfrag, amsfonts, bm, color, upgreek, mathtools, datetime2}
\usepackage{extsizes}%%%for extra size
\usepackage[dvipsnames]{xcolor}
\usepackage[shortlabels]{enumitem}
\usepackage{subcaption}

\DeclareMathAlphabet{\mathpzc}{OT1}{pzc}{m}{it}

\newtheorem{thm}{Theorem}[section]

\newtheorem{theorem}[thm]{Theorem}
\newtheorem{Ex}[thm]{Example}

\newtheorem{rmk}[thm]{Remark}

\newtheorem{Cor}[thm]{Corollary}
\newtheorem{Prop}[thm]{Proposition}
\newtheorem{prop}[thm]{Proposition}

\newtheorem{lem}[thm]{Lemma}
\newtheorem{Lem}[thm]{Lemma}
\newtheorem{lemma}[thm]{Lemma}
\newtheorem{dfn}[thm]{Definition}

\numberwithin{equation}{section}

\newcommand{\be}{begin{equation}}

\newcommand{\z}{\mathbb{Z}}
\newcommand{\N}{\mathbb{N}}
\renewcommand{\c}{\mathbb{C}}
\newcommand{\R}{\mathbb{R}}

\newcommand{\cal}{\mathcal}

\newcommand{\SL}{\operatorname{SL}}

\newcommand{\bp}{\begin{pmatrix}}
\newcommand{\ep}{\end{pmatrix}}

\renewcommand{\bp}{{\rm bp}}

\newcommand{\rok}{{\mathcal O}_\mathbb{K}}
\newcommand{\rol}{{\mathcal O}_\mathbb{L}}
\newcommand{\roku}{{\mathcal O}_\mathbb{K}^\times}

\newcommand{\norm}[1]{\lVert #1 \rVert}
\newcommand{\norme}[1]{\lVert #1 \rVert_{e}}
\newcommand{\abs}[1]{\lvert #1 \rvert}

\newcommand{\op}{\operatorname}

\renewcommand{\deg}{\operatorname{deg}}

\newcommand{\cl}[1]{\overline{#1}}

\newcommand{\Lie}{{\rm Lie}}

\renewcommand{\be}{\begin{equation}}
\newcommand{\ee}{\end{equation}}
\newcommand{\ben}{\begin{enumerate}}
\newcommand{\een}{\end{enumerate}}

\newcommand{\F}{\operatorname{F}\!}

\newcommand{\FM}{\F X}

\newcommand{\RFM}{\op{RF}X}

\newcommand{\core}{\text{core}}

\newcommand{\hull}{\op{hull}}

\newcommand{\bq}{\begin{quotation}}
\newcommand{\eq}{\end{quotation}}

\newcommand{\m}[1]{\mathbb{ #1}}

\newcommand{\PGLt}{\operatorname{PGL}_2}
\newcommand{\GLt}{\operatorname{GL}_2}
\newcommand{\PSLt}{\operatorname{PSL}_2}

\newcommand{\qpo}{\mathbb{Q}_p(\omega)}
\newcommand{\qpl}{\mathbb{Q}_p(\lambda)}
\newcommand{\qp}{\mathbb{Q}_p}
\newcommand{\qpob}{\overline{\mathbb{Q}_p(\omega)}}
\newcommand{\qpb}{\overline{\mathbb{Q}_p}}

\newcommand{\zp}{\mathbb{Z}_p}
\newcommand{\zpo}{\mathbb{Z}_p(\omega)}

\newcommand{\zpl}{\mathbb{Z}_p(\lambda)}

\newcommand{\omg}{\omega}
\newcommand{\al}{\alpha}

\newcommand{\ut}{\tau}

\newcommand{\xx}{\times}

\newcommand{\vtxg}{{\bf{V}}(T_{G})}

\newcommand{\ax}{\operatorname{Axis}}

\newcommand{\dis}{\displaystyle}

\newcommand{\pn}[1]{\abs{#1}_p}

\newcommand{\vep}{\varepsilon}

\newcommand{\qua}[1]{(a_{#1},b_{#1},k_{#1},u_{#1})}

\newcommand{\opp}{O^{+}}
\newcommand{\onn}{O^{-}}

\newcommand{\bv}{{\mathbf v}}
\newcommand{\bu}{{\mathbf u}}
\newcommand{\bvv}{{\mathbf v}^{*}}
\newcommand{\bw}{{\mathbf w}}

\newcommand{\ggl}{\mathbb{G}(\mathbb{L})}

\newcommand{\diag}{\op{diag}}

\newcommand{\nrc}[1]{N_{r}^{C,S_{\Gamma}}(#1)}

\makeatletter
\newcommand{\dbloverline}[1]{\overline{\dbl@overline{#1}}}
\newcommand{\dbl@overline}[1]{\mathpalette\dbl@@overline{#1}}
\newcommand{\dbl@@overline}[2]{%
  \begingroup
  \sbox\z@{$\m@th#1\overline{#2}$}%
  \ht\z@=\dimexpr\ht\z@-2\dbl@adjust{#1}\relax
  \box\z@
  \ifx#1\scriptstyle\kern-\scriptspace\else
  \ifx#1\scriptscriptstyle\kern-\scriptspace\fi\fi
  \endgroup
}
\newcommand{\dbl@adjust}[1]{%
  \fontdimen8
  \ifx#1\displaystyle\textfont\else
  \ifx#1\textstyle\textfont\else
  \ifx#1\scriptstyle\scriptfont\else
  \scriptscriptfont\fi\fi\fi 3
}
\makeatother

\newcommand{\scc}{\mathbb{S}_{p}^{1}}
\newcommand{\mup}{\mathbb{\mu}_{p^{2}-1}}

\newcommand{\gm}{\gamma}
\newcommand{\Gm}{\Gamma}

\makeatletter
\newenvironment{equationate}{%
 \itemize
 \let\orig@item\item
 \def\item{\orig@item[]\refstepcounter{equation}
 (\theequation) 
 \hspace{5pt}
 \def\item{\hfill(\theequation)\orig@item[]\refstepcounter{equation}}}
}{%
 \enditemize
}
\makeatother

\begin{document}

\title[]{$\PGLt(\qp)$-orbit closures on a $p$-adic homogeneous space of infinite volume}

\author{Jinho Jeoung and Seonhee Lim}

\begin{abstract} 
Let $\mathbb{K}$ be an unramified quadratic extension of $\qp$ for a fixed $p>2$. 
Projective general linear groups $G=\PGLt(\mathbb{K})$ and $H=\PGLt(\qp)$ act transitively on Bruhat-Tits trees $T_G$ and $T_H$, respectively. 
We identify $G/H$ with the set of $H$-subtrees $G.T_{H}$.
Let $\Gamma$ be a Schottky subgroup such that $\Gamma\backslash T_{G}$ is infinite volume and has an additional condition named high-branchedness,
and let $\Lambda$ be its limit set.

We classify $\Gamma$-orbits in $G/H$.
Let $C=g_{C}H\in G/H$.
As a generalization of Ratner's theorem, if $\Gamma\backslash g_{C}.T_{H}$ meets the convex core of $\Gamma\backslash T_{G}$, then the $\Gamma$-orbit of $C$ is either dense or closed in
$
\cal{C}_{\Lambda}=\{g H: \partial(g.T_{H})\cap\Lambda\neq\varnothing\}$.

\end{abstract}
\maketitle

\section{Introduction}

The celebrated Ratner's measure classification theorem and orbit closure theorem, stand on homogeneous spaces with finite volume over a local field of characteristic zero \cite{R1,R2,R3,R4,R5,MT}.
For positive characteristics, we refer to \cite{EM,MG,M}.
Thanks to her theorems,
unipotent dynamics have been an important tool for classifying the closures of unipotent orbits.
As in her paper \cite{R1}, a direct consequence of her theorem is the classification of $\SL_{2}(\R)$-orbits in a homogeneous space $\Gm\backslash\SL_{2}(\c)$, where $\Gm$ is a lattice in $\SL_{2}(\c)$.
In this case, every $\SL_{2}(\R)$-orbit is either closed or dense in the homogeneous space.
Since $\Gm\backslash\SL_{2}(\c)$ can be viewed as a frame bundle of the 3-manifold $\Gamma\backslash\m{H}^{3}$ of finite volume, it is an example of geometric rigidity.
We remark that this is a 3-dimensional case of a more general result of Shah \cite{Sha}.
Shah proved his theorem in a purely topological way while Ratner used her measure classification theorem.

On the other hand, developments in geometric rigidity within homogeneous spaces of infinite volume have recently emerged \cite{MMO,MMO2,KO,LO,Ka,TZ}.
The technical difficulty of the infinite volume case arises from the fact that, 
unlike the finite volume case, we cannot directly use the recurrence of unipotent orbits into the given compact region as almost every unipotent orbit (with respect to Haar measure) escapes any compact region (see Remark \ref{rmk1}, and also \cite{MMO}). 
Nevertheless, Mohammadi, McMullen, and Oh introduced the concept of $K$-thickness to capture scarce recurrence of unipotent orbits into a compact region \cite{MMO}.
Among 3-manifolds of infinite volume, 
convex cocompact acylindrical 3-manifolds provide scarce recurrence allowing for the application of their method, known as the unipotent blowup.
Thanks to their insights, unipotent dynamics remains a crucial tool for the classification of orbit closures.

In this article,
we investigate the classification of $\PGLt(\qp)$-orbit closures on a non-Archimedean homogeneous space $\Gm\backslash\PGLt(\m{K})$ of infinite volume.
We need a suitable discrete subgroup $\Gm$, a non-Archimedean analog of a convex cocompact acylindrical Kleinian subgroup, i.e., a subgroup in $\PGLt(\m{K})$, where $\m{K}$ is the unramified quadratic extension of $\qp$ for odd prime $p$.
We will call such discrete subgroups convex cocompact highly-branched Schottky groups (see Definition \ref{def:acyl}).

%\subsection{Acylindrical Schottky subgroups}\label{ss:acy}
Let $p$ be an odd prime number.
For an unramified quadratic extension of the $p$-adic field $\qp$, we express
$$
\mathbb{K}=\qpo=\{x+\omg y:x,y\in\qp \},
$$
where $\omg$ is a root of a degree 2 irreducible polynomial $f(t)\in\mathbb{F}_{p}[t]$.
The rings of integers of $\qp$ and $\qpo$ are $\zp$ and $\zpo$, respectively.
Let
\begin{align*}
	G=\PGLt(\qpo)\cong\op{Isom}^{+}(T_{G}), 
	\hspace{10pt} 
	H=\PGLt(\qp)\cong\op{Isom}^{+}(T_{H}),
\end{align*}
where $T_{G}$ and $T_{H}$ are the \emph{Bruhat-Tits trees} of $G$ and $H$, respectively.
Here, $T_{G}$ and $T_{H}$ have degree $p^{2}+1$ and $p+1$, respectively.
We regard $T_{H}$ as a subtree of $T_{G}$ (see Section \ref{ss:bt}).
For a vertex $\bv$ in a tree or a graph $T$, denote the degree of $\bv$ in $T$ by $\deg_{T}(\bv)$.
%$\qpo^\times=\qpo-\{0\}$, and $\qp^\times=\qp-\{0\}$. 
%The groups $G$ and $H$ act transitively on the set of edges of their \emph{Bruhat-Tits trees} $T_{G}$ and $T_{H}$, which are regular trees of degree $p^{2}+1$ and $p+1$, respectively.
%$$
%\partial T_{G}\cong \m{P}^{1}(\qpo)\cong\qpob:=\qpo\cup\{\infty\}
%$$
%
% identified with projective spaces $\m{P}^{1}(\qpo)\cong\qpob:=\qpo\cup\{\infty\}$ and $\m{P}^{1}(\qp)\cong\qpb:=\qp\cup\{\infty\}$, respectively. 

Let $\Gamma$ be a finitely generated torsion-free discrete subgroup of $\PGLt(\qpo)$ of infinite covolume.
Since a subgroup $\Gamma$ of $\PGLt(\qpo)$ for which $\Gamma\backslash T_{G}$ of infinite volume is virtually Schottky (\cite[Proposition 1.7]{Lub}, and \cite[Proposition 3.7]{HeH}),
we may assume that $\Gamma$ is a Schottky subgroup of $G$.
Thus, let $\Gamma$ be a Schottky subgroup (see Definition \ref{dfn:schottky}) of $G$ such that $\Gamma\backslash G$ has infinite volume.
We remark that $\Gamma\backslash G$ has infinite volume if and only if the quotient graph $\Gamma\backslash T_{G}$ is an infinite graph \cite{Ser}.
Let $\Lambda\subset\partial T_{G}$ be the limit set of $\Gamma$ and 
$$S_{\Gamma}=\hull(\Lambda)\subset T_{G}$$
be the \emph{convex hull} of $\Lambda$ (see Section \ref{ss:sch}).
%For $\bv,\bw\in T_{G}$, we let $d(\bv,\bw)$, the distance between $\bv$ and $\bw$, be the number of edges between $\bv$ and $\bw$.

Additionally, we need the following subsets of $\zpo$:
let $\mup$ be the multiplicative group of ($p^{2}-1$)-th roots of unity in $\zpo$, and let
\begin{equation}\label{scc}
	\scc=\{x+\omg y\in 1+p\zpo:x^{2}-\omg^{2}y^{2}=1\}.
\end{equation}
Let hyperbolic elements $\gamma_{1},\gamma_{2},\cdots,\gamma_{n}\in G$ be generators of $\Gamma$.
Note that
$\Gamma$ consists of hyperbolic elements only.
Hence, for any $\gm\in\Gamma$, there exists a diagonal element 
\begin{equation}\label{asubtau}
\diag(p^{n_{\gm}}a_{\gm},1)\in G
\end{equation}
that is conjugate to $\gm$ with $n_{\gm}\in \z^{\xx}$ and $a_{\gm}\in\zpo^{\xx}$.
{
Here, $a_{\gamma}$ is decomposed as $r_{\gm}\Theta_{\gm}$ for some $r_{\gm}\in 1+p\zp$ and $\Theta\in \mup\xx\scc$ (see Section \ref{ss:polar}.)
Let $\cal{F}$ be a Schottky fundamental domain (see Definition \ref{dfn:schottky} and Proposition \ref{prop:Schottky}).
Define 
$$\cal{F}':=\{\bv\in S_{\Gamma}\cap\cal{F}: N_{1}(\bv)\not\subseteq\cal{F}\},$$
where $N_{1}(\bv)$ is the $1$-neighborhood of $\bv$ in $T_{G}$ (see (\ref{nbd})).

\begin{dfn} \label{def:acyl}
For a Schottky subgroup $\Gamma$,
we say that $\Gamma$ is \emph{highly-branched} if
\begin{enumerate}
\item
every vertex $\bv\in S_{\Gamma}$ satisfies
$
	\deg_{S_{\Gamma}}(\bv)\geq p^{2}-p+2
$
and for any $\bu,{\bu}'\in \cal{F}'$, there exists a vertex $\bw$ on the geodesic between $\bu$ and $\bu'$ such that 
$$\deg_{S_{\Gamma}}(\bw)\geq p^{2}-p+3.$$
%
%every axis of a Schottky basis element of $\Gamma$ has a vertex with degree at least $p^{2}+1-(p-2)$ in its Schottky fundamental domain.

\item there exists $\gm\in\Gamma$ such that the set
$
	\{\Theta_{\gm}^{n}:n\in\z_{\geq0}\}
$
is a dense subset of $\mup\xx\scc$.

\end{enumerate}
\end{dfn}
We will explain the conditions in Definition \ref{def:acyl} after the main theorems.

\subsection{Main theorems}\label{ss:main}
Identify the visual boundaries $\partial T_{G}$ and $\partial T_{H}$ with $\qpob:=\qpo\cup\{\infty\}$ and $\qpb:=\qp\cup\{\infty\}$, respectively. 
The groups $G$ and $H$ act transitively on $\qpob$ and $\qpb$ by Möbius transformations (see Section \ref{ss:bt}).
Let 
$$\cal{C}=\{C=g.\cl{\qp}:g\in G\}$$
be the set of circles equipped with the topology defined by the Hausdorff distance:
for any $C_{i}=g_{i}.\cl{\qp}\in \cal{C}$ with $g_{i}\in G$, where $i=1,2$,
$$
	d(C_{1},C_{2})=\max\left\{
	\sup_{y\in\cl{\qp}} d(C_{1},g_{2}.y),
	\sup_{x\in\cl{\qp}} d(g_{1}.x,C_{2})
	\right\}.
$$
In addition, let
$$
	\cal{C}_{\Lambda}=\{D\in\cal{C}:D\cap\Lambda\neq\varnothing\}.
$$
For every $C\in\cal{C}$, we let 
$\Gamma^{C}=\op{Stab}_{\Gamma}(C).$

Let $X:=\Gm\backslash T_{G}$, and $\pi:T_{G}\rightarrow X$ be the projection.
Define its \emph{convex core} as 
$$\core(X)=\Gamma\backslash S_{\Gamma}.$$
Since $\Gamma$ has no parabolic element, $X$ has no cusp.
It follows that $\Gamma$ is convex-cocompact, i.e. $\core(X)$ is compact.
Call each connected component in the non-compact part $\core(X)^{c}$ an $\emph{end}$ of $X$.
Note that this terminology is different from the end of Gromov hyperbolic spaces.
Each end is a tree, which is ($p^{2}+1$)-regular at all vertices except the vertex attached to the convex core.
\begin{thm} \label{thm:main}
Let $\Gamma$ be a convex cocompact highly-branched Schottky subgroup of $G$.
For $C\in\cal{C}$, one of the following holds:
\begin{enumerate}
    \item $C\cap \Lambda=\emptyset$, and $\Gamma C$ is discrete in ${\cal{C}}$; 
    \item $C\cap \Lambda=C$, $\Gamma^C$ is conjugate to a cocompact subgroup of $\PGLt(\qp)$, and $\Gamma C\subset {\cal{C}}$ is discrete; 
    \item $C\cap \Lambda \neq\varnothing$ nor $C$, $\Gamma^C$ is conjugate to a convex cocompact subgroup of $\PGLt(\qp)$, and $\Gamma C\subset {\cal{C}}$ is discrete; 
    \item $C\cap \Lambda \neq\varnothing$ nor $C$, and $\overline{\Gamma C}=\cal{C}_{\Lambda}.$
\end{enumerate}
\end{thm}

Identifying $g.\qpb$ with $gH$ in $G/H$, 
we reformulate $\cal{C}_{\Lambda}$ as a subset of $G/H$ by
$$\cal{C}_{\Lambda}=\{g_{D}H:\partial(g_{D}.T_{H})\cap\Lambda\neq\varnothing\}.$$
Then, Theorem \ref{thm:main} classifies $H$-orbits in $\Gamma\backslash G$ by the duality between $\Gamma$-orbits in $G/H$ and $H$-orbits in $\Gamma\backslash G$ (see Section \ref{ss:bt}).

\begin{Cor}
If $C\cap\Lambda\neq\varnothing$, then $C$ is either closed or dense in $\cal{C}_{\Lambda}$.
\end{Cor}
\begin{proof}
By Proposition \ref{prop:denseorbit}, if $\Gamma C$ is not closed in $\cal{C}_{\Lambda}$, then $\Gamma C$ is dense in $\cal{C}_{\Lambda}$.
\end{proof}

Furthermore, we can interpret Theorem \ref{thm:main} for the quotient graph $X=\Gamma\backslash T_{G}$.
For any tree or graph $T$, we denote the set of its vertices by $\textbf{V}(T)$.
\begin{thm}\label{thm:sub}
Define a map $\iota_{g}:\textbf{V}(T_H)\rightarrow X$ by $\iota_{g}(\bv)=\pi(g.\bv)$ for every $g\in G$. Let $P=\cl{\iota_{g}(T_H)}$. 
Then either:
\begin{enumerate}
    \item $P$ is a regular subtree of $X$ of degree $p+1$ 
    in an end of $X$; or
    \item $P$ is a closed subgraph in $\core(X)$; or
    \item $P$ is an infinite graph of $X$ with $\core(P)=\core(X)$ ; or
    \item $P=X$.
\end{enumerate}
\end{thm}

McMullen, Mohammadi, and Oh classified $\PSLt(\R)$-orbits (hence, the classification of geodesic planes follows) in $\Gamma\backslash\PSLt(\c)$ when $\Gamma\subset\PSLt(\c)$ is rigid acylindrical Kleinian \cite{MMO} and general acylindrical Kleinian \cite{MMO2}.
The classification of $\PSLt(\R)$-orbits was obtained by purely topological arguments by using the minimality of the action of the complex unipotent subgroup (proved by Ferte or Winter \cite{Fe,Win})
on the entire space with finite volume \cite{Sha}, and on the convex core of the space with infinite volume \cite{MMO,MMO2}. 
On the other hand, Kwon proved the analog of Winter's result for Bruhat-Tits trees \cite{Kwon}. 
However, we need the WSG (weighted spectral gap) property for $\Gamma$ and a $\Gamma$-invariant potential function to apply the result of Kwon.
Since the WSG property does not imply Definition \ref{def:acyl}, we do not assume it for $\Gamma$.
We will use a different strategy to obtain the analog of Winter's result.
We first prove that the closure of a non-closed $\PGLt(\qp)$-orbit in $\Gamma\backslash G$ that meets $\core(X)$ contains an orbit of the unipotent subgroup of $\PGLt(\qpo)$.
Next, we topologically prove the minimality of such unipotent orbits on the set of all $\PGLt(\qp)$-orbits in $\Gamma\backslash G$ that meet $\core(X)$.
Through the duality between right $H$-orbits in $\Gamma\backslash G$ and left $\Gamma$-orbits in $G/H$, we use the convergence of $\Gamma$-orbits in $G/H$ (see Section \ref{ss:dense}).

We remark that the geometrically finite quotient of $\PSLt(\c)$ is proved by Benoist and Oh \cite{BO}. 
For the classification of geodesic planes outside the convex hull, 
we refer to a recent work by Torkaman and Zhang \cite{TZ}.

\noindent\begin{rmk}\label{rmk2}
{\rm
Theorem 1.5 of \cite{MMO} which corresponds to Theorem \ref{thm:main} contains 5 cases.
In the real hyperbolic space, 
the core of the quotient $\core(\Gamma\backslash\m{H}^{3})$ for
a Zariski dense convex cocompact acylindrical Kleinian subgroup $\Gamma$ is
a compact manifold with a boundary. 
In particular, $\core(\Gamma\backslash\m{H}^{3})$ has non-empty interior. 
Theorem 1.5 of \cite{MMO} includes the case when 
a geodesic plane fills an end of $M=\Gamma\backslash \m{H}^{3}$, which is the case (4) of Theorem 1.5 of \cite{MMO}.

In this regard, our definition of high-branchedness is somewhat stronger than the rigid acylindricality in the real case.
Case (4) of Theorem 1.5 of \cite{MMO}, where $|C\cap\Lambda|=1$ does not occur in our case as the intersection cannot be a finite set by Proposition \ref{prop:infsubset}.
}
\end{rmk}
\noindent\begin{rmk}\label{rmk1}
{\rm
Now let us explain the necessity of the conditions in Definition \ref{def:acyl}.
Condition (1) restricts the degree of vertices in the invariant subtree $S_{\Gamma}$.
We will use the unipotent blowup lemma analogous to that of \cite{MMO}.
Denote the subtree $g.T_{H}\cap S_{\Gamma}$ by $S_{g}$.
Let $\cal{U}$ be the $\qp$-unipotent subgroup of $G$.
To obtain the unipotent blowup lemma, we desire a certain recurrence property that the $(-\infty)$-point of a geodesic in $S_{g}$ returns infinitely many times into $\partial S_{g}$ by $\cal{U}$-action within a bounded time gap.
More precisely,

\begin{equationate}\label{recur}
\item if $\ell=(g_{\ell}.0,g_{\ell}.\infty)$ is a given geodesic in $S_{g}$ and $u_{i}\in\cal{U}$ such that $g_{\ell}u_{i}.0\in S_{g}$, then we want to show the existence of $u_{j}\in\cal{U}-\{u_{i}\}$ such that $g_{\ell}u_{j}.0\in S_{g}$ and
$$\abs{\log_{p}\norm{u_{i}-e}-\log_{p}\norm{u_{j}-e}}$$
 is bounded.
\end{equationate}
 
To ensure the recurrence property described in (1.3), $\partial g.T_{H}\cap\partial S_{\Gamma}$ must be infinite unless it is empty.
Suppose that there exists some constant $K$
such that for any vertex $\bv\in S_{g}$ with $\op{deg}_{S_{g}}(\bv)\geq 3$,
there exists $\bw\in S_{g}$ such that $\op{deg}_{S_{g}}(\bw)\geq 3$ and the number of edges between $\bv$ and $\bw$ is less than $K$.
The existence of such $K$ implies that
$\partial g.T_{H}\cap\partial S_{\Gamma}$ is infinite
(see Proposition \ref{prop:infsubset} for details).
Then, we can prove that for every $l\in\z$, there exists $x\in (p^{-(K+l)}\zp-p^{-l}\zp)$ so that $g_{\ell}u_{x}.0\in \partial S_{g}$, where $u_{x}$ is a unipotent element defined by $x$ (see Proposition \ref{prop:k-thick}).
%
%
%Moreover, this condition allows $g_{\ell}.0$, the $(-\infty)$-point of a geodesic $\ell$, to return to $S_{g}$ regularly by the right $\cal{U}$-action.
%In particular, if $\{u_{i}\}_{i=0}^{\infty}$ is a sequence of unipotents elements that let $g_{\ell}.0$ return to $S_{g}$, and that $\norm{u_{i}-e}<\norm{u_{i+1}-e}$,
%then $\abs{\log_{p}\norm{u_{i}-e}-\log_{p}\norm{u_{i+1}-e}}<K$ for all $i\in\z_{\geq0}$.

In \cite{MMO}, the limit set of a convex cocompact acylindrical subgroup is a Sierpi\'nski carpet.
The limit set being a Sierpi\'nski carpet implies
an important property called $K$-thickness (see Definition \ref{def:k-thick} for  $K$-thickness in our setting), which ensures the recurrence property in (1.3).
However, a direct analog of Sierpi\'nski carpet in $\partial T_{G}$ is not obvious since $\partial T_{G}$ is fractal and totally discrete. 
Thus, instead of dealing with the boundary $\partial T_{G}$, we use the tree $T_{G}$ to give a condition, which is condition (1).

Without condition (1) in Definition \ref{def:acyl}, the boundary $\partial S_{g}$ may be a finite set.
If $\partial S_{g}$ is finite, we cannot ensure the unipotent blowup lemma.
We provided such non-example in Example \ref{ex:subgroup}.

	Condition (2) in Definition \ref{def:acyl} follows from Zariski density of $\Gamma$ in the real case.
	However, in our case, Zariski density does not ensure the density of
$
	\bigcup_{\gm\in\Gamma} \{\Theta_{\gm}^{n}\}
$
in $\mup\xx\scc$ (see Section \ref{ss:zar}).
In this regard, we need condition (2), a condition stronger than the Zariski density of $\Gamma$.
We will use this condition to prove Proposition \ref{prop:denseorbit}.
}\end{rmk}

%%%%%%%%%%%%%%%%%%%%%%%%%%%%%%%%%%%%

\subsection{Acknowledgements}
We thank Hee Oh for suggesting the problem and helpful discussions.
We also thank Amir Mohammadi, Yves Benoist, Taehyeong Kim, Omri Solan, and Konstantin Andritch for their insightful discussions.
Jinho Jeoung and Seonhee Lim are supported by the National Research Foundation of Korea, project number NRF-2020R1A2C1A01011543.
 
\section{Preliminaries}
 In this section, we gather preliminary facts about $p$-adic dynamics, group actions on the Bruhat-Tits tree, which is a $p$-adic analog of a symmetric space, frames, axes of hyperbolic elements, and examples of convex cocompact highly-branched subgroups.
 
\subsection{$p$-adic field and its unramified quadratic extension}\label{ss:p-adic}
Let $\qp$ be the $p$-adic field for an odd prime $p$, i.e.,
$$
	\qp=\left\{\sum_{i=n}^{\infty} a_{i}p^{i}:a_{i}\in\mathbb{F}_{p},
	\, n\in\z \right\}.
$$
Its residue field is $\mathbb{F}_{p}\simeq\zp/p\zp$.
Let $\nu_{p}:\qp\rightarrow\z$ be the $p$-adic valuation of $\qp$ and $\pn{\cdot}$ be the $p$-adic norm.
If $a=\sum_{i=n}^{\infty} a_{i}p^{i}$, then
$$
	\pn{a}=p^{-\nu_{p}(a)}=p^{-n}.
$$
Using the $p$-adic norm, the ring of integers and the unit ring of integers of $\qp$ are characterized as follows.
\begin{align*}
	&\zp=\{a\in\qp:\pn{a}\leq1\},
	&\zp^{\xx}=\{a\in\qp:\pn{a}=1\}.
\end{align*}
As in the introduction, let $\m{K}$ be a degree $2$ unramified extension of $\qp$ and $\rok$ be the ring of integers of $\m{K}$.
We remark that an unramified extension of $\qp$ of each degree is unique (see Section 6.5 of \cite{Gou}).
Since $\m{K}$ is unramified, we can choose the uniformizer of $\m{K}$ to be $p$.
Hence we can use $\nu_{p}$ and $\pn{\cdot}$ on $\m{K}$ as well.
The residue field of $\m{K}$ is $\m{F}_{p^{2}}\simeq\rok/p\rok$. 
Moreover, there is a degree $2$ irreducible polynomial $f(t)\in\m{F}_p[t]$ and its root $\omega\notin\m{F}_p$ so that we can express
$$
	\m{K}=\qpo=\left\{\sum_{i=n}^{\infty} 
	(a_{i}+\omega b_{i})p^{i}:a_{i},b_{i}\in\zp/p\zp,
	\, n\in\z \right\},
$$
\begin{align*}
	\rok=\zpo=\{x\in\m{K}:\pn{x}\leq1\},
	\hspace{3pt}{\rm{and}}\hspace{5pt}
	\roku=\zpo^{\xx}=\{x\in\m{K}:\pn{x}=1\}.
\end{align*}
Here,
for $x=a+\omega b\in\qpo$,
$$
	\pn{x}=\max
	\{
	\left|a\right|_{p},
	\left|b\right|_{p}
	\}.
$$

\subsubsection{$p$-adic exponential and trigonometric functions}
In this subsection, we state properties of $p$-adic exponential and trigonometric functions (see \cite[Chapter II, Section 4]{VVZ} for more details).

Remark that $\m{F}_{p^{f}}$ contains $i=\sqrt{-1}$ if and only if $p^{f}\not\equiv 3 \text{(mod 4)}$ for $f\in\z_{\geq1}$.
For all odd prime $p$, we obtain $p^{2}\equiv 1\text{(mod 4)}$.
Since our residue field is $\m{F}_{p^{2}}$, we have $i\in\qpo$ for all $p$ by Hensel's lemma.

There are 2 cases for $\omg$.
The first case is $\omg=\al i$ for some $\al\in\qp$ so that $\omg i\in\qp$, and the second case is $\omg i\notin\qp$.
Let $\lambda=\omg_{p} i$, where $\omg_{p}=\omg$ if $\omg i\notin\qp$, and $\omg_{p}=1$ if $\omg =\al i$.
If we let $\lambda=\omg_{p} i$, then $\qpl=\qpo$, and $\zpl=\zpo$.
Moreover, there is an isomorphism between $\mup\cdot\scc$ defined on $\qpl$ and $\mup\cdot\scc$ defined on $\qpo$.
Thus, for computational convenience, let us use $\qpl$ instead of $\qpo$ until Subsection \ref{ss:polar}.

For $x\in p\zpl$, define
$$
e^{x}:=\sum_{n=0}^{\infty}\frac{x^{n}}{n!},
\hspace{10pt}
\sin x:=\sum_{n=0}^{\infty}(-1)^{n}\frac{x^{2n+1}}{(2n+1)!},
\hspace{10pt}
\cos x:=\sum_{n=0}^{\infty}(-1)^{n}\frac{x^{2n}}{(2n)!}.
$$
These series are convergent whenever $\pn{x}\leq p^{-1}$.
We have the following properties of $\cos x$ and $\sin x$ on $\qpl$:
\begin{align}\label{trigo}
\cos^{2}x+\sin^{2}x=1,
\hspace{20pt}
e^{ix}=\cos x+i\sin x.
\end{align}
Note that $\omg_{p}^{-1}\sin (\omg_{p} x)$ and $\cos (\omg_{p} x)-1$ are both in $p\zp$ when $x\in p\zp$.
The following property for $e^{x}$ is still valid on $\qpl$:
\begin{align}\label{exp}
e^{x+y}=e^{x}e^{y}.
\end{align}

%Let $i$ be a root of the irreducible polynomial $t^{2}+1$.
%We remark that if 
%$p\equiv 3 \hspace{3pt} (\op{mod} 4)$,
%then $i\notin\qp$ and if 
%$p\equiv 1 \hspace{3pt} (\op{mod} 4)$,
%then $i\in\qp$.
%By the uniqueness of unramified extension of $\qp$, we may choose $\omg$ any root not of the form $\sqrt{dp},$ for $d \in \mathbb F_{q}$.
%Choose a square-free $r\in\m{F}_{p}$ and let $\omg=i{\nu}$, where $\nu=1$ if $i\notin\qp$ and $\nu=\sqrt{r}$ if $i\in\qp$.
%In either case, we have the following properties of $\cos x$ and $\sin x$ on $\qpo$:
%\begin{align}\label{trigo}
%\cos^{2}x+\sin^{2}x=1,
%\hspace{20pt}
%e^{ix}=\cos x+i\sin x.
%\end{align}
%Note that $\nu^{-1}\sin (\nu x)$ and $\cos (\nu x)-1$ are both in $p\zp$ when $x\in p\zp$.
%The following property for $e^{x}$ is still valid:
%\begin{align}\label{exp}
%e^{x+y}=e^{x}e^{y}.
%\end{align}

\subsubsection{$p$-adic polar coordinates}\label{ss:polar}
%The decomposition of $\qpl^{\xx}$ (as $\qpo^{\xx}$) is a key ingredient to prove the behavior of dense orbits in Section 4.
%Let $\mup$ be the $(p^{2}-1)$-th root of unity in $\qpl$.
%First, $\zpl$ is decomposed by 
%$$
%	\zpl^{\xx}\cong\mup\xx(1+p\zpl)
%$$
%as a multiplicative topological group \cite[Section 4.3]{Ro}.

Define $f:p\zp\rightarrow p\zp$ by $f(x)=\omg_{p}^{-1}\sin(\omg_{p} x)$.
Here, $\sin(\omg_{p} x)=\omg_{p} f(x)$ and $\arcsin (y)$ is well-defined on $p\zpl$ as a power series (see (4.23) of \cite{VVZ}).
It follows that $f(x)$ is bijective on $p\zp$.
Now consider the following equation:
\begin{equation}\label{eqn:trigo}
a^{2}-\lambda^{2}f(x)^{2}=a^{2}+\sin^{2}(\omg_{p} x)=1.
\end{equation}
%
%Now define the norm map $N:\qpo^{\xx}\rightarrow\qp$: for $a,b\in\qp$,
%$$
%	N(a+\omg b)=a^{2}-\omg^{2}b^{2}.
%$$
%Remark that 
%$
%N((a+\omg b)(c+\omg d))
%=N(a+\omg b)N(c+\omg d)
%$
%for $a,b,c,d\in\qp$.
%
%Thus $\scc$ in (\ref{scc}) is reformulated as
%$
%	\scc=\{ y\in1+p\zpo:N(y)=1\}.
%$
%Define $f:p\zp\rightarrow p\zp$ by $f(x)=\omg^{-1}\sin(\omg x)=\sum_{n=0}^{\infty}(-1)^{n}\omg^{2n}\frac{x^{2n+1}}{(2n+1)!}$.
%Here, $\sin(\omg x)=\omg f(x)$ and $\arcsin (y)$ is well-defined on $p\zp(\omg)$ as a power series (see (4.23) of \cite{VVZ}).
%It follows that $f(x)$ is bijective on $p\zp$.
%Thus, if we choose $y=a+\omg b\in \scc$, there exists $x\in p\zp$ so that $y=a+\omg f(x)$ and
%$$
%1=N(y)=a^{2}+(i\omg f(x))^{2}=a^{2}+i^{2}\sin^{2}(\omg x).
%$$
Since $\cos(\omg_{p} x)\in 1+p\zp$ and $-\cos(\omg_{p} x)\notin 1+p\zp$,
we obtain
$a=\cos(\omg_{p} x)$ by (\ref{trigo}).
Moreover,
since $f(x)$ is bijective on $p\zp$, any $a+\lambda b\in 1+p\zpl$ satisfying $a^{2}-\lambda^{2}b^{2}=1$ 
is expressed by
$$
a+\lambda b=a+\lambda f(x)=\cos(\omg_{p} x)+i\sin(\omg_{p} x)=e^{i\omg_{p} x}
$$
for some $x\in p\zp$.
Thus $\scc$ in (\ref{scc}) is reformulated by
\begin{equation}\label{scceq}
	\scc=\{ e^{i\omg_{p} \theta}: \theta\in p\zp\}.
\end{equation}

\begin{Lem}\label{lem:important}
Let $e^{i\omg_{p} \theta}\in\scc$, where $\theta=p\cdot\sum_{i=0}^{\infty}{x_{i}}p^{i}\in p\zp$.
Then $\langle e^{i\omg_{p}\theta}\rangle$ is dense in $\scc$ if and only if $x_{0}\neq 0$.
\end{Lem}
\begin{proof}
For $z\in\zp$, 
the set $\{nz\}_{n=1}^{\infty}$ is dense in $\zp$ if and only if $z\in\zp^{\xx}$ since $\z_{\geq1}$ is dense in $\zp$.
Since
$\langle e^{i\omg_{p}\theta} \rangle = \{e^{i\omg_{p}(n\theta)}:n\in\z_{\geq1}\}$ by (\ref{exp}),
the claim follows.
\end{proof}

$
	\qp^{1}=\{p^{n}(1+y):n\in\z, y\in p\zp\}.
$
From Theorem 1 in \cite{Ko}, we have $\qpl^{\xx}\cong\qp^{1}\xx\mup\xx\scc$ as multiplicative topological groups.
Thus, we will use $\qp^{1}\cdot\mu_{p^{2-1}}\cdot\scc$ rather than the product space.
Now if
$\theta=p\cdot\sum_{i=0}^{\infty}{x_{i}}p^{i}\in p\zp$ with $x_{0}\neq0$ and
$\zeta$ is a generator of $\mup$, then $\cl{\langle \zeta e^{i\omg_{p} \theta}\rangle}=\mup\cdot \scc$.
We can interpret $\langle \zeta e^{i\omg_{p} x} \rangle$ as an analog of an irrational rotation.

\subsubsection{Zariski dense subgroup of $G$}\label{ss:zar}

In this subsection, we explain why Zariski density condition on $\Gamma$ does not ensure high-branchedness of $\Gamma$ in Definition \ref{def:acyl} (2).
In the beginning, we deal with $G'=\GLt(\qpo)$ instead of $G=\PGLt(\qpo)$ to find an example in $G$.

We consider $G$ and $G'$ as $\qp$-Lie groups (cf. $\op{SL}_{2}(\c)$ and $\PSLt(\c)$ as real Lie groups).
A subgroup $\Gamma$ of $G$ (resp. $G'$) is \emph{Zariski dense} if there is no proper $\qp$-algebraic subvariety in $G$ (resp. $G'$) which contains $\Gamma$.

Let $\Gamma'$ be a Schottky subgroup of $G'$.
Let $\diag(p^{n_{1}}a_{1},p^{n_{2}}a_{2})$ 
be conjugate to a hyperbolic element $\gamma\in\Gm'$ with $a_{j}\in\zpo^{\xx}$ for $j=1,2$.
Decompose 
$$a_{j}=r_{j}\zeta_{j} e^{i\omg \theta_{j}}=r_{j}\Theta_{j}
$$ for some $r_{j}\in 1+p\zp$, $\zeta_{j}\in\mup$, and $\theta_{j}\in p\zp$.
Let 
$$\Omega(\Gamma')=
\{\Theta_{1}\Theta_{2}^{-1} \in \mup\cdot \scc:
\diag(p^{n_{1}}\Theta_{1}, p^{n_{2}}\Theta_{2})\sim\gamma\in\Gamma' \}.$$
By definition, 
for the projectivized group $\Gm\subset\PGLt(\qpo)$ of $\Gm'$,
we have
$\Omega(\Gamma)=\Omega(\Gamma')$.
We want $\Omega(\Gamma)$ to be dense in $\mup\cdot\scc$.
However, we can show the existence of $\Gamma'\subset G'$ being Zariski dense but $\Omega(\Gm')$ is far from dense.
Then, by projectivizing $\Gm'$ to $G$, we have a Zariski dense subgroup $\Gm$ such that $\Omega(\Gm)$ is not dense.

Before we state the Proposition, 
define a projection $\pi_{\theta}:\Gm'\rightarrow p\zp$ by $\pi_{\theta}(\gm)=\theta_{1}-\theta_{2}$,
where $\gm\in\Gm'$ is defined as the above.

\begin{prop}
There exists Zariski dense Schottky group $\Gamma$ in $G$ such that $\Omega(\Gamma)$ is not dense in $\mup\cdot\scc$.
\end{prop}

\begin{proof}
We start with the following statement in
\cite[Section 30.4]{Hum}:
let $S$ be a maximal proper (Zariski) closed subgroup of $G'$.
Then, either the identity component $S^{o}$ of $S$ is reductive or $S$ is parabolic.

In our case, reductive subgroups are conjugates of the diagonal subgroup and parabolic subgroups are conjugates of the upper triangular subgroup.
Now for $j=1,2$, define elements $\theta_{j}$ and $\theta'_{j}$ of $p\zp$ by $\theta_{j}=p^{k_{j}}\cdot\sum_{n=0}^{\infty}x_{j,n}p^{n}$ and
$\theta'_{j}=p^{k'_{j}}\cdot\sum_{n=0}^{\infty}x'_{j,n}p^{n}$,
where $k_{j},k'_{j}\in\z_{>1}$.
In addition, choose $\zeta_{j},\zeta'_{j}$ as generators of $\mup$, and $n_{j},n'_{j}\in\z$.
Let $\Theta_{j}=\zeta_{j}e^{i\omg\theta_{j}}$, 
$\Theta'_{j}=\zeta'_{j}e^{i\omg\theta'_{j}}$, and
\begin{align*}
\gamma=g\diag(p^{n_{1}}\Theta_{1},
p^{n_{2}}\Theta_{2})g^{-1}, 
{\text{ and }}
\gamma'=g'
\diag(p^{n'_{1}}\Theta'_{1},p^{n'_{2}}\Theta'_{2})g'^{-1},
\end{align*}
where $g.\infty,g'.\infty,g.0$, and $g'.0$ are all mutually distinct.
Let $\Gm'=\langle\gm,\gm'\rangle$.
Note that $\Omega(\Gm')$ is infinite since 
$\pi_{\theta}(\gm^{n})=n\pi_{\theta}(\gm)=n(\theta_{1}-\theta_{2})$ for $n\in\z$, thus, by the proof of Lemma \ref{lem:important}.
Observe that 
$$\det(\gamma\gamma')=p^{n_{1}+n_{2}+n'_{1}+n'_{2}}(\zeta_{1}\zeta_{2}\zeta'_{1}\zeta'_{2})e^{i\omg(\theta_{1}+\theta_{2}+\theta'_{1}+\theta'_{2})}.$$
Thus, $\pi_{\theta}(\gm\gm')$ is an element of $p^{\min\{k_{j},k'_{j}\}}\zp$.
It follows that $\pi_{\theta}(\Gm')\subset p^{\min\{k_{j},k'_{j}\}}\zp$.
Hence, $\Omega(\Gm')$ cannot be dense in $\mup\cdot\scc$ as $\min\{k_{j},k'_{j}\}>1$.

We now claim that $\Gm'$ is Zariski dense in $G'$. 
First, since $\Gamma'$ consists of hyperbolic elements, it is not conjugate to the parabolic subgroup.
Now note that $\langle \diag(p^{n_{1}}\Theta_{1},
p^{n_{2}}\Theta_{2})\rangle$ is Zariski dense in the diagonal subgroup of $G'$ by our choice of $\zeta_{j}$, and $\theta_{j}$ for $j=1,2$.
Hence, the Zariski closure of $\Gamma'$ contains a conjugate of the diagonal subgroup.
Moreover, since $g_{1}\neq g_{2}$, the subgroup $\Gamma'$ cannot be conjugate to a subgroup of the diagonal subgroup.
Thus, by maximality of the diagonal subgroup, the Zariski closure of $\Gamma'$ is $G'$.
Then, the projectivized group $\Gm\subset G$ of $\Gamma'$ is indeed a Zariski dense subgroup of $G$ with $\Omega(\Gm)$ not dense in $\mup\cdot\scc$.
\end{proof}

Remark that in contrast to our case,
for a Zariski dense Kleinian subgroup $\Gamma$, the analog of $\Omega$ is dense in $\m{S}^{1}$ (see \cite[Proposition 4.4]{MMO}).

\begin{figure}[h!]
  \includegraphics[width=0.5\linewidth]{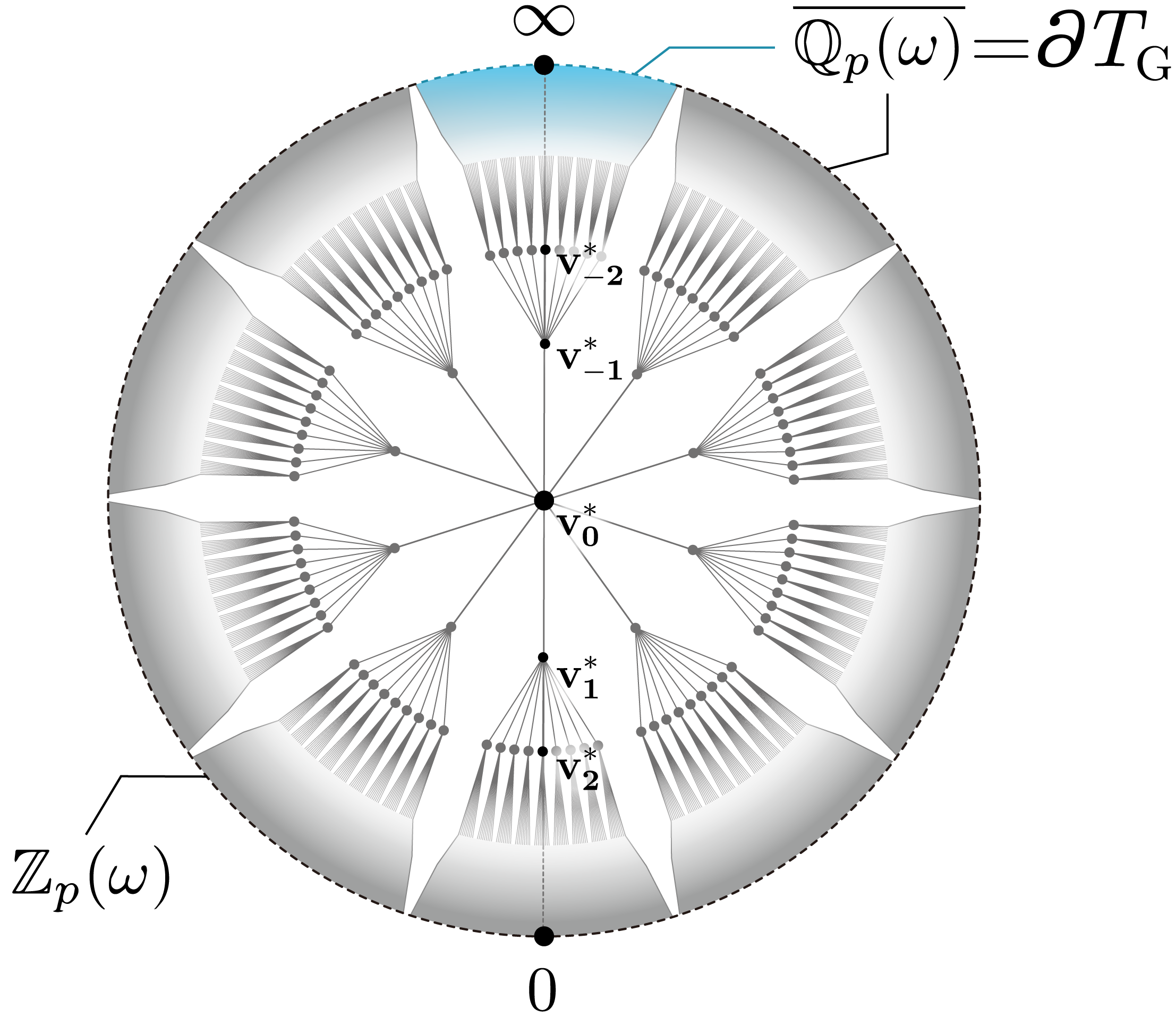}
  \caption{The Bruhat-Tits tree of $T_{G}$ of $G$ when $p=3$ and 
  			$\omega=\sqrt{2}$.
We express the subset $\zpo$ of $\partial T_{G}$ by the blue dashed line.
The blue and black dashed lines represent the entire boundary of $T_{G}$.}
  \label{fig:bt1}
\end{figure}
 
\subsection{The Bruhat-Tits tree}\label{ss:bt}
For detailed references, we refer to \cite{Ser}, \cite{BPP},  and \cite{KP}.
	Let $\m{L}$ be either $\qpo$ or $\qp$, and $\rol$ be its ring of integers.
Let $\ggl$ be the projective general linear group $\PGLt(\m{L})$.
Then $\ggl$ acts transitively on its Bruhat-Tits tree. 
To define the Bruhat-Tits tree $T_{\ggl}$ of $\ggl$, define the set of vertices of 
$T_{\ggl}$ by ${\bf{V}}(T_{\ggl})=\PGLt(\m{L})/\PGLt(\rol)$. 
Define the reference geodesic from $0$ to $\infty$ to be the sequence of vertices $\{\bvv_{j}\}_{j=-\infty}^{\infty}$, where
$$
	\bvv_{j}=\begin{bmatrix} p^{j} & 0 \\ 0 & 1 \end{bmatrix}
	=\begin{pmatrix} p^{j} & 0 \\ 0 & 1 \end{pmatrix}.\PGLt(\cal{O}_{\m{L}})
	\in{\bf{V}}(T_{\ggl}).
$$
Let $\bvv_{0}$ be the origin of $T_{\m{G}(\m{L})}$.
On the other hand, we can consider ${\bf{V}}(T_{\ggl})$ as the space of lattices of $\m{L}^{2}$.
Lattices $\bv$ and $\bw$ in ${\bf{V}}(T_{\ggl})$ are adjacent if and only if either $p\bv\subsetneq \bw\subsetneq \bv$ or $p\bw\subsetneq \bv\subsetneq \bw$ as sets.
For example,
{\small
$
	\bw=\begin{bmatrix} p & a \\ 0 & 1 \end{bmatrix}
$
}
is an adjacent vertex of $\bvv_{0}$ for any $a\in\rol/p\rol$. 
Let ${\bf{E}}(T_{\ggl})$ be the set of (unoriented) edges of $T_{\ggl}$. 
Edges are defined by pairs of two adjacent vertices. 
The graph $T_{\ggl}=(\textbf{V}(T_{\ggl}),\textbf{E}(T_{\ggl}))$ is indeed a tree.

The visual boundary of $T_{\ggl}$ is identified with $\m{P}^{1}(\m{L})=\cl{\m{L}}$.
Let $\infty\in\cl{\m{L}}$ be
{\small
$\dis
\lim_{j\rightarrow\infty}
\begin{bmatrix}
	p^{-j}&0\\0&1
\end{bmatrix}
=\infty
$}.
Let $\al=\sum_{i=m}^{\infty}a_{i}p^{i}\in\m{L}$.
For $n\geq m$, define $\al_{n}=\sum_{i=m}^{n}a_{i}p^{i}$
and
\begin{equation}\label{vertexname}
	\bv(\al_{n})=
	\begin{bmatrix} p^{n+1} & \al_{n} \\ 0 & 1 \end{bmatrix}.
\end{equation}
Then, $\al\in\m{L}$ is identified with a sequence of vertices, 
$$\{\cdots,\bv_{m-1}^{*},\bv_{m}^{*},\bv(\al_{m}),\bv(\al_{m+1}),\cdots\},$$
which is a bi-infinite geodesic from $\infty$ to $\al$.
If $\al=0$, then let $m=-\infty$.

%By using the right $\PGLt(\rol)$-action on $\ggl$ and the equivalence relation on the projective general linear group, 
%every element $\bv$ in ${\bf{V}}(T_{\ggl})$ can be represented by
%\begin{equation}\label{vertexname}
%	\bv(a)=
%	\begin{bmatrix} p^{m(\bv)+n(\bv)} & a \\ 0 & 1 \end{bmatrix}
%\end{equation}
%for $a=\sum_{i=m(\bv)}^{m(\bv)+n(\bv)-1}a_{i}p^{i}$ with $n(\bv)\geq 1$.
%Here, $|m(\bv)|+n(\bv)=d(\bv_{0}^{*},\bv(a))$ and $n(\bv)=d((0,\infty),\bv(a))$.

Now, let $G=\PGLt(\qpo)$ and $H=\PGLt(\qp)$.
If we let $\bv=g_{\bv}.\PGLt(\zpo)=[g_{\bv}]$ for some $g_{\bv}\in G$, 
we define the $G$-action on $\vtxg$ by
$$
	g.\bv=g.[g_{\bv}]=gg_{\bv}.\PGLt({\zpo})=[gg_{\bv}]
$$
for all 
{\small $g=
\begin{pmatrix}
	a&b\\c&d
\end{pmatrix}
\in G$}.
Considering $\bv(\al_{n})$ as a lattice in $\qpo^{2}$, the above action is expressed by
$$
g.\left(
\zpo\binom{p^{n+1}}{0}+\zpo\binom{\al_{n}}{1}
\right)
=
\zpo\binom{p^{n'}}{0}+\zpo\binom{\frac{a\al_{n}+b}{c\al_{n}+d}}{1}
$$
for some $n'\in\z$.
Hence, 
$$
g.\al=
\{\cdots,g.\bv_{m-1}^{*},g.\bv_{m}^{*},g.\bv(\al_{m}),g.\bv(\al_{m+1}),\cdots\}
=\frac{a\al+b}{c\al+d},
$$
so that the action of $G$ on $\partial T_{G}$ coincides the action of $G$ on $\cl{\qpo}$, which is the Möbius transformation.

The distance $d(\bv,\bw)$ on the tree $T_{G}$ is given by the number of edges between vertices $\bv$ and $\bw$.
For any $\bv,\bw\in{\bf{V}}(T_{G})$ and $\al\in\cl{\qpo}$, denote the \emph{geodesic ray} from $\bv$ to $\al$ by $[\bv,\al)$, 
and the geodesic segment from $\bv$ to $\bw$ by  $[\bv,\bw]$.
The $n$-neighborhood of a vertex $\bv$ in $T_{G}$ is defined by
\begin{equation}\label{nbd}
	N_{n}(\bv)=\{\bw\in\vtxg: d(\bv,\bw)\leq n\}.
\end{equation}
For each $\bv\in{\bf{V}}(T_{G})$, adjacent vertices of $\bv$ are important to us
as they give a direction at $\bv$.
Fix
$a=\sum_{i=n-k}^{n-1}a_{i}p^{i}$ for $a_{i}\in\zpo/p\zpo$, $k\geq 1$, and let
$
	\bv(a)
$
be as in (\ref{vertexname}).
Then there are $(p^{2}+1)$-adjacent vertices, 
$\bv({a+bp^{n}})$ for $b\in\zpo/p\zpo$, and $\bv({a'})$ for $a'=\sum_{i=n-k}^{n-2}a_{i}p^{i}$.
For vertices $\bv({a+bp^{n}})$ with $b\in\zp/p\zp$ and $\bv({a'})$, we say that they are in \emph{$\qp$-direction} at the vertex $\bv({a})$. 
Additionally, we say that a subtree is an \emph{$H$-subtree} if it is of the form $g.T_{H}$ for some $g\in G$.

Let us define a \emph{branch} of $T_{G}$.
Let $\bv$ be a vertex is $T_{G}$ and $\bw$ be an adjacent vertex of $\bv$.
A subset 
$B=\{{\bf{x}}\in\vtxg : d({\bf{x}},\bw)<d({\bf{x}},\bv)\}$
 of $T_{G}$ is a rooted subtree of $T_{G}$ and we call $B$ a branch containing $\bw$ at $\bv$.
There are ($p^{2}+1$)-branches at each vertex $\bv\in\vtxg$ which are $(p^{2}+1)$-regular at every vertex except the root vertex adjacent to $\bv$, which has degree $p^{2}$.

\subsubsection{The matrix groups and their actions on the Bruhat-Tits tree}
The usual max norm of $\GLt(\qpo)$ is not well-defined on $G=\PGLt(\qpo)$ as its value depends on the representative of the element of $G$.
To define a metric on $G$, first, let us fix the following subset of $\GLt(\qpo)$ by
$$
	G_{1}:=\left\{
	(a_{ij})
	: a_{22}=1
	\right\},
\hspace{10pt}
	G_{2}:=
	\left\{
	(a_{ij})
	: a_{12}=1, a_{22}=0
	\right\}.
$$
It is easy to check that any element of $G$ has a unique representative in either $G_{1}$ or $G_{2}$.
Remark that 
{\small $\begin{pmatrix} 0&1\\1&0 \end{pmatrix}\in \GLt(\qp)$}
sends every element of $G_{2}$ into $G_{1}$, thus
representatives of $H$-coset in $G/H$ can be chosen from $G_{1}$.
Now 
for any $g,h\in G$, choose $g',h'\in G_{1}\sqcup G_{2}$ such that $[g']=g$ and $[h']=h$ in $G$.
From the max norm of $\GLt(\qpo)$, we define a metric on $G$ by
$$
	d(g,h)=\norm{g'-h'}=\max_{1\leq i,j\leq2}\{ \pn{g'_{ij}-h'_{ij}}\}.
$$
We will use the following in Section 3 to measure the distance between $g\in G$ and the identity: for $g\in G$ and $g'\in G_{1}\sqcup G_{2}$ such that $[g']=g$,
$$
	\norme{g}:=\norm{g'-e}
	=\max\{\pn{g'_{11}-1},\pn{g'_{12}},\pn{g'_{21}},\pn{g'_{22}-1}\}.
$$

Recall that we denoted a diagonal matrix by $\diag(a_{1},a_{2})$.
For unipotent matrices, we let 
{\small$u_{t}=\begin{pmatrix}
	1 & t \\ 0 & 1
\end{pmatrix}$}.
Now, we define the following subgroups of $G$ which have crucial roles in our paper.
\begin{align}\label{matrices}
\begin{split}
\cal{A}&=\left\{
\diag(a,1)
:a\in\qp^{\times} \right\},
	\hspace{10pt}
\cal{U}=\left\{u_{t}
:t \in \qp \right\},\\
\cal{V}&=\left\{u_{\omg t}
:t \in \qp \right\},
	\hspace{20pt}
\cal{N}=\left\{
u_{s}
:s \in \qpo \right\}. 
\end{split}
\end{align}

\subsubsection{Circles in the boundary $\partial T_{G}$}
	We define circles in $\qpob$. 
The action of $G$ on $T_{G}$ and on $\partial T_{G}$ have the following relation. For any $g\in G$,
$$
	\partial(g. T_{G})=g.(\partial T_{G})=g.\qpb.
$$ 
Comparing our setup with the real hyperbolic case by the following,
	\begin{align*}
		T_{G}\longleftrightarrow \m{H}^{3},\hspace{5pt} 
		\partial T_{G}&\cong\qpob\longleftrightarrow \cl{\c}\cong\m{S}^{2},
	\\
		T_{H}\longleftrightarrow\m{H}^{2},\hspace{5pt}
		\partial T_{H}&\cong\qpb\longleftrightarrow \m{S}^{1}\subset\cl{\c},
	\end{align*}
it is reasonable to call $g.\qpb$ for $g\in G$ a circle in $\partial T_{G}$.
We remark that $g.\partial T_{H}$ is not a $p$-adic metric circle in $\cl{\qpo}$ but only a fractal subset for every $g\in G$.
Then, $\cal{C}=\{C=g. \qpb: g\in G\}$
can be considered as the set of circles in $\qpob$. 
Note that $g.\qpb=g.\partial T_{H}=gH.x$ for any $x\in\qpb$. 
Thus we can identify $\cal{C}$ with $G/H$.
We will often use this identification.

\subsubsection{Frames and the frame bundle of $T_{G}$}
	Our goal of this subsection is to define the renormalized frame bundle $\RFM$ of $X=\Gamma\backslash T_{G}$ in $\Gamma\backslash G$.
Before defining $\RFM$, we will define the frames and geodesics corresponding to elements in $G$. 

	Every (bi-infinite) geodesic in $T_{G}$ is uniquely determined by a pair of distinct boundary points.
	Let $\cal{G}$ be the set of triple points on the boundary:
$$
	\cal{G}=\{(x,y,z):x,y,z\hspace{5pt}
	\textrm{are mutually distinct elements of}\hspace{5pt}\cl{\qpo}\}.
$$
We call $\cal{G}$ the frame bundle over $T_{G}$.
The Möbius action of $G$ on $\cal{G}$ is simply transitive, thus we identify $G$ with $\cal{G}$ via the map $\Phi:G\rightarrow \cal{G}$  given by
$$
	\Phi(g)=g.(0,\infty,1)=(g.0,g.\infty,g.1).
$$ 
	For $g\in g_{\bv}.\PGLt(\zpo)$, the intersection of three geodesics defined by any two points of $\{g.0,g.\infty,g.1\}$ is the vertex $\bv$.
It is natural to regard $(g.0,g.\infty,g.1)$ as a frame at $\bv$.
From now on, we will identify the frame $(g.0,g.\infty,g.1)$ with $g$.
With this identification, the set of all frames at $\bv$ is $g_{\bv}.\PGLt(\zpo)$.
Hence, we regard $G$ as the frame bundle over $T_{G}$.

	We have the following identification:
\begin{align*}
    G/\PGLt(\zpo) \simeq 
    \textbf{V}(T_{G}).
\end{align*}
Note that $\PGLt(\zpo)$ is the maximal compact subgroup of $G$.
The above identification is an analog of the relation of $\PSLt(\c)$ with the frame bundle of $\mathbb{H}^3$.

Define the \emph{limit set} of $\Gamma$ by
$
	\Lambda=\cl{\Gamma.\bv\cap\partial T_{G}}
$
for some $\bv\in\vtxg$, which is independent of the choice of $\bv$ \cite[Section 3]{SC}.
Recall that we defined $S_{\Gamma}$ as the convex hull of $\Lambda$, i.e., the smallest closed subset of $T_{G}$ that contains every infinite geodesic in $T_{G}$ connecting two distinct elements of $\Lambda$.
It follows that $\partial S_{\Gamma}=\Lambda$ and $\ax(\gamma)\subset S_{\Gamma}$ for all $\gamma\in\Gamma$.
We will discuss axes of hyperbolic elements in the next subsection.

	\begin{dfn}\label{def:rfx}
		Denote the frame bundle over $X=\Gamma\backslash T_{G}$ by $\FM=\Gamma\backslash G$.
		The renormalized frame bundle $\RFM$ is defined as
$$
\RFM=\{[g]\in\FM:g.0,g.\infty\in\Lambda\}\subset\Gamma\backslash G.
$$
	\end{dfn}
 We remark that if $g\in\RFM$, then the projection of geodesic $(g.0,g.\infty)$ is bounded in $X$ (actually, in $\core(X)$) since $(g.0,g.\infty)$ is in $S_{\Gamma}$.
We also remark that $\RFM$ is compact in $FX$ since the projection of $\RFM$ to $X$ is exactly $\core(X)$, which is a finite graph (see Theorem 5 in Section 1.5 in Chapter 2 of \cite{Ser}).

\begin{rmk}\label{rmk:coset}
{\rm
Note that $\FM/H=\Gamma\backslash G/H$ and $\Gamma\backslash \cal{C}$ are identified.
The frame bundle $\FM$ and the set of circles $\cal{C}$ have quotient topologies induced from the topology of $G$.
Consequently, topologies of $\FM$ and $\cal{C}$ induce the same topology of $\Gamma\backslash G/H$.
We remark that $H$-action on $\FM$ is minimal if and only if $\Gamma$-action on $\cal{C}$ is minimal
(see Section 1.7 of \cite{GN}).
}
\end{rmk}

\subsection{Schottky subgroups of $\PGLt(\m{K})$}\label{ss:sch}
	For $n\geq2$, let $\{\gamma_1,\gamma_2,\cdots,\gamma_n\}$ be a subset of $G$ consisting of hyperbolic elements.
Recall that each $\gamma_{i}$ is conjugate to a diagonal matrix of the form
$
	d_{i} = 
    \begin{pmatrix} p^{-n_i}a_{i} & 0 \\ 0 & 1 
    \end{pmatrix}
$,
where $n_{i}\in\N$ is the translation length of $\gamma_i$ and $a_{i}\in\zpo^{\xx}$.
Set $\gamma_{i}=g_{i}d_{i}g_{i}^{-1}$ for some $g_{i}\in G$.
Denote the \emph{axis} of $\gamma_i$ by $\ax(\gamma_i)$. 
Let $\gamma^{-}_i$ and $\gamma^{+}_i$ be the repelling and attracting fixed points of $\gamma_i$, respectively. 
Then, $g_i.0=\gamma_i^{-}$ and $g_i.\infty=\gamma_i^{+}$, and $\ax(\gamma_i)=(\gamma_i^{-},\gamma_i^{+})$. 

	Label vertices of $\ax(\gamma_{i})$ from $\gamma_{i}^{-}$ to $\gamma_{i}^{+}$ by $\{\bv^{i}_{j}:j\in\z\}$ so that $\op{d}(\bv^{i}_{j},\bv^{i}_{j+k})=k$ for all $k\in\z$.
Let
$$
	\opp_{i}=
	\{\bv\in\vtxg:d(\bv,\bv^{i}_{n_{i}+1})<d(\bv,\bv^{i}_{n_{i}})\}
$$
and
$$
	\onn_{i}=
	\{\bv\in\vtxg:d(\bv,\bv^{i}_{0})<d(\bv,\bv^{i}_{1})\}.
$$
Note that $\gamma_{i}.(\onn_{i})=T_{G}-\opp_{i}$
and $\gamma_{i}.(T_{G}-\onn_{i})=\opp_{i}$.

\begin{dfn}\cite[Definition 1.4]{Lub}\label{dfn:schottky}
The group
$
	\Gamma=\langle \gamma_{1},\gamma_{2}\cdots\gamma_{n} \rangle, \hspace{3pt}
	n\geq 2
$
is called a \emph{Schottky group} if there exists a label of vertices of axes of $\gamma_{i}$'s such that $\{\opp_1,\cdots,\opp_{n},\onn_{1},\cdots,\onn_{n}\}$ are mutually disjoint.
The set of generators $\{\gamma_{1},\gamma_{2},\cdots,\gamma_{n}\}$ is called a \emph{Schottky basis}.
\end{dfn}

The following proposition will be used frequently.

\begin{prop}\cite[Proposition 1.6]{Lub}\label{prop:Schottky}
	Let $\Gamma=\langle\gamma_{1},\cdots,\gamma_{n}\rangle$, $n\geq2$ be a Schottky group. 
	 We have the following:
	\begin{enumerate}
		\item $\Gamma$ is a discrete group.
		\item Every element in $\Gamma$ is hyperbolic.
		\item $\Gamma$ is a free group with elements of Schottky basis as generators. 
		\item the set $\cal{F}=\vtxg-\bigcup_{i=1}^{n}(\opp_{i}\cup\onn_{i})$ is a 	fundamental domain for the $\Gamma$-action.
	\end{enumerate}
\end{prop}

	A fundamental domain in Proposition \ref{prop:Schottky} is called \emph{a Schottky fundamental domain}.
Since the convex core of $X$ is compact, the \emph{diameter} of $\core(X)$ defined by
$
	\op{diam}(\core(X))=\max\{d(\bv,\bw):\bv,\bw\in S_{\Gamma}\cap\cal{F}\}
$
is finite.

\subsection{Axes of hyperbolic elements and the limit set}\label{ss:ax}
\begin{figure}[h!]
  \includegraphics[width=0.4\linewidth]{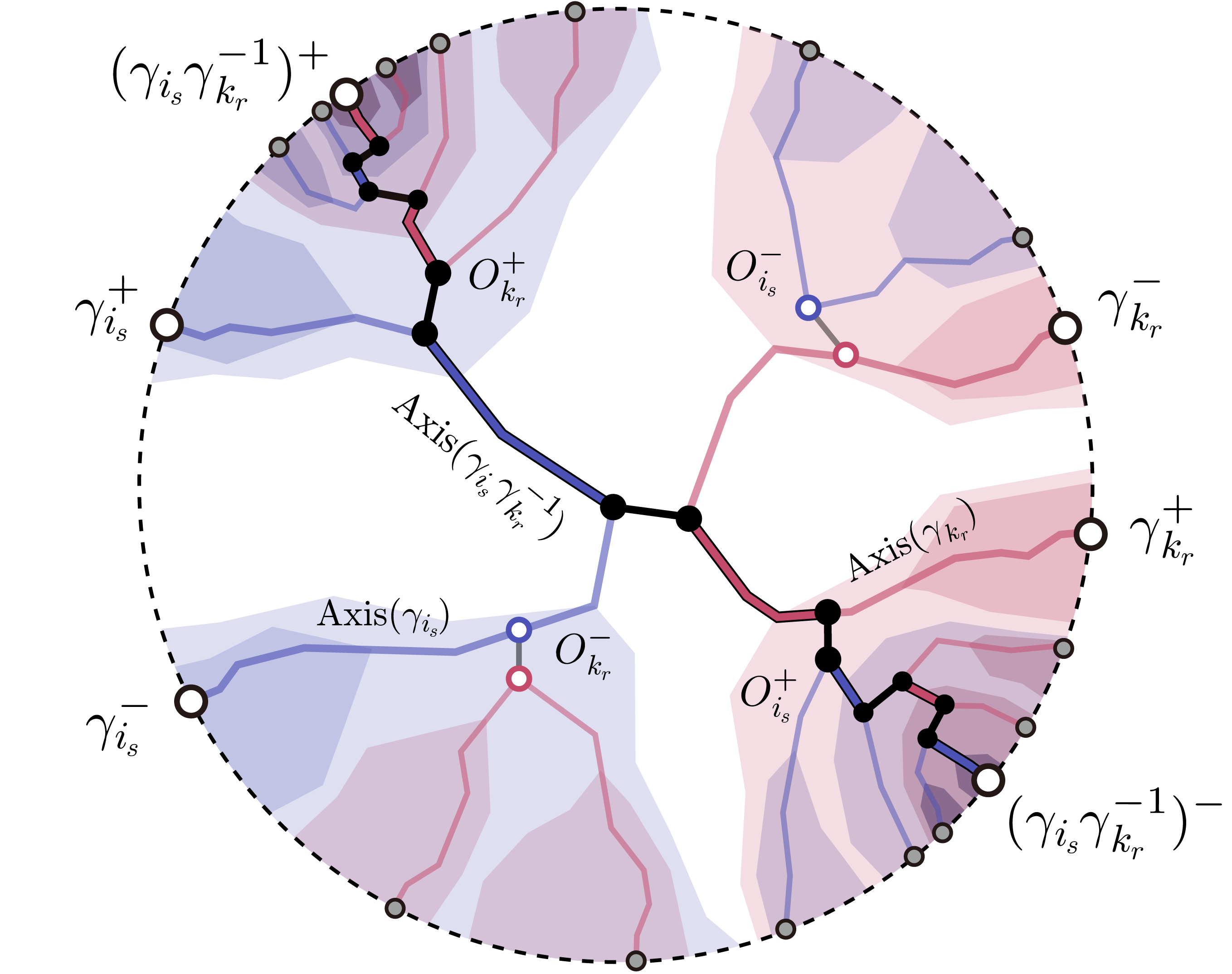}
  \caption{
  Example of the axis of $\gamma_{i_{s}}\gamma_{k_{r}}^{-1}$ for some $\gamma_{i_{s}}$ and $\gamma_{k_{r}}$.}
  \label{fig:axis}
\end{figure}
Let us choose Schottky basis elements $\gamma_{i}\in\Gamma$ with $\ax(\gamma_{i})=(\gamma_{i}^{-},\gamma_{i}^{+})$, and $\opp_{i},\onn_{i}$ for $i=1,\cdots,n$ with a Schottky fundamental domain $\cal{F}$ as in Section \ref{ss:sch}.
Note that $\Lambda=\partial S_{\Gamma}\subset\partial\left(\bigcup_{i=1}^{n}(\opp_{i}\cup\onn_{i})\right)$
since $\partial\cal{F}$ contains no end points of any axes of elements of $\Gamma$.

We first claim that if $\bigcup_{i=1}^{n}\ax(\gamma_{i})$ is a connected set, then we have $\cal{F}\cap\bigcup_{i=1}^{n}\ax(\gamma_{i})=\cal{F}\cap S_{\Gamma}$.
Choose $x,y\in\Lambda$.
We may assume that $x\in\partial\opp_{i}$ and $y\in\partial\opp_{j}$ for some $i,j\in\{1,2,\cdots n\}$.
Since $(x,y)\subset S_{\Gamma}=\hull(\Lambda)$, we have $[\bv_{n_{i}}^{i},\bv_{n_{j}}^{j}]\subset(x,y)$, where $\bv_{n_{i}}^{i}\in\ax(\gamma_{i}),\bv_{n_{j}}^{j}\in\ax(\gamma_{j})$ as in Section \ref{ss:sch}.
By connectedness of $\bigcup_{i=1}^{n}\ax(\gamma_{i})$, it follows that $[\bv_{n_{i}}^{i},\bv_{n_{j}}^{j}]\subset\bigcup_{i=1}^{n}\ax(\gamma_{i})$. 
Therefore, every vertex in $\cal{F}\cap S_{\Gamma}$ is contained in some axis.

We now assume that $\bigcup_{i=1}^{n}\ax(\gamma_{i})$ has more than one connected component.
We claim that for every vertex $\bv$ in $\cal{F}\cap S_{\Gamma}$, there exists $\gamma\in\Gamma$ such that $\bv\in\ax(\gamma)$.
Let $A_{i}=\bigcup_{j\in I_{i}}\ax(\gamma_{j})$ be connected components of $\bigcup_{i=1}^{n}\ax(\gamma_{i})$
for each $i=1,2,\cdots,m$.
Now fix $A_{i}$ and consider 
$$d(A_{i},A_{j}):=\min\{d(\bv,\bw):\bv\in A_{i},\bw\in A_{j}\}.$$
Let $A_{k}\neq A_{i}$ be a connected component minimizing $d(A_{i},\cdot)$.
Let $\bw_{i}\in A_{i}$ and $\bw_{k}\in A_{k}$ be the vertices such that $d(\bw_{i},\bw_{k})=d(A_{i},A_{k})$.
Then, $\bw_{i}$ and $\bw_{k}$ must be in $\cal{F}$.
Otherwise, for instance if $\bw_{i}\notin\cal{F}$, then $\bw_{i}\in \cal{O}_{l}^{+}\cup\cal{O}_{l}^{-}$ for some $l\in\{1,2,\cdots n\}$.
By definition of $\cal{F}$ and $A_{i}$, it follows that $\bw_{i}\in \ax(\gamma_{l})$.
Thus, for any $\bw\in\ax(\gamma_{l})\cap\cal{F}$, we have $d(\bw,A_{k})<d(\bw_{l},A_{k})$, which is a contradiction.

Note that there exist $\gamma_{i_{s}},\gamma_{k_{r}}\in\{\gamma_{1},\cdots\gamma_{n}\}$ such that $\bw_{i}\in\cal{F}\cap\ax(\gamma_{i_{s}})$ and $\bw_{k}\in\cal{F}\cap\ax(\gamma_{k_{r}})$.
Furthermore, if we let $\gamma=\gamma_{i_{s}}\gamma_{k_{r}}^{-1}$, then $\gamma^{\ell}.\bw_{i}\in\opp_{i_{s}}$ and $\gamma^{-\ell}.\bw_{k}\in\opp_{k_{r}}$ for all $\ell\in\z_{>0}$.
Hence, $\gamma^{+}\in\partial\opp_{i_{s}}$ and $\gamma^{-}\in\partial\opp_{k_{r}}$.
Consequently, it follows that $[\bw_{i},\bw_{k}]\subset\ax(\gamma)$.
Any vertex in $\cal{F}\cap S_{\Gamma}$ is contained in either such $A_{i}$'s or $[\bw_{i},\bw_{k}]$'s since $S_{\Gamma}=\hull(\Lambda)$.
Hence, the claim follows.

%Since $S_{\Gamma}=\hull(\Lambda)$, 
%Choose a vertex $\bv$ from $(\cal{F}\cap S_{\Gamma})-\bigcup_{i=1}^{n}\ax(\gamma_{i})$.
%
%For some $i,j\in\{1,2,\cdots,n\}$, we have $\bv\in[\bw_{i},\bw_{j}]$, where $\bw_{i}\in\ax(\gamma_{i}),\bw_{j}\in\ax(\gamma_{j})$ are mutually distance-minimizing to the other axis that one does not belong to.
%Note that for any $\bu_{i}\in\ax(\gamma_{i})$ and $\bu_{j}\in\ax(\gamma_{j})$, $[\bw_{i},\bw_{j}]\subseteq[\bu_{i},\bu_{j}]$ since $[\bw_{i},\bw_{j}]$ is the path between $\ax(\gamma_{i})$ and $\ax(\gamma_{j})$.
%Hence, $[\bw_{i},\bw_{j}]\subseteq[\gamma_{i}^{n}\gamma_{j}^{-n}.\bw_{i},\gamma_{j}^{n}\gamma_{i}^{-n}.\bw_{j}]$ for all $n\in\z$.
%Letting $n\rightarrow\infty$, we have $[\bw_{i},\bw_{j}]\subset\ax(\gamma_{i}\gamma_{j}^{-1})$.
%Thus the claim follows.

Now for any $\bv\in S_{\Gamma}$, there exists $\ut\in\Gamma$ such that $\ut^{-1}.\bv\in \cal{F}\cap S_{\Gamma}$.
By our claim, there exists $\gamma\in\Gamma$ such that $\ut^{-1}.\bv\in\ax(\gamma)$.
Thus $\bv\in\ax(\ut\gamma\ut^{-1})$.
Moreover, for any geodesic segment $[\bv,\bw]$ in $S_{\Gamma}$, there exist $\ut_{1},\ut_{2}\in\Gamma$ such that $\bv\in\ax(\ut_{1})$, and $\bw\in\ax(\ut_{2})$.
Even if $\ut_{1},\ut_{2}$ are not Schottky basis elements, we can similarly define $O_{\ut_{i}}^{\pm}$ for $i=1,2$.
Then, for $i=1,2$, there exists $n_{i}\in\z$ such that $\ut_{i}^{n_{i}}.[\bv,\bw]\subset O_{\ut_{i}}^{+}$.
It follows that $\ax(\ut_{1}^{n_{1}}\ut_{2}^{-n_{2}})$ contains $[\bv,\bw]$.
%Let $\bu_{1}\in\ax(\ut_{1}),\bu_{2}\in\ax(\ut_{2})$ be mutually distance-minimizing to the other axis that one does not belong to.
%There exist $n_{1},n_{2}\in\z$ such that 
%$[\bv,\bw]\subseteq[\rho_{1}^{n_{1}}.\bu_{1},\rho_{2}^{n_{2}}.\bu_{2}]$,
%or $[\bv,\bw]\subseteq[\rho_{1}^{n_{1}}.\bu,\rho_{2}^{n_{2}}.\bu]$, respectively.
%similarly, $[\bv,\bw]\subseteq\ax(\rho_{1}^{n_{1}}\rho_{2}^{-n_{2}})$ in both case for some $n_{1},n_{2}\in\z$.
Since $\bv,\bw$ are arbitrary vertices in $S_{\Gamma}$, we conclude that any geodesic segment in $S_{\Gamma}$ is contained in $\ax(\ut)$ for some $\ut\in\Gamma$.

\begin{prop}\label{prop:axisapp}
For any pair of two distinct points $\al,\beta\in\Lambda$, there is a sequence $\{\ut_{i}\}_{i=0}^{\infty}$ in $\Gamma$ such that $\ax(\ut_{i})\rightarrow(\al,\beta)$ as $i\rightarrow\infty$.
\end{prop}
\begin{proof}
Let $\al,\beta\in\Lambda$.
From Section \ref{ss:bt},
recall that 
$
	N_{m}(\bv_{0}^{*})
$
for some $m\in\z_{\geq1}$ is
the $m$-neighborhood of the origin $\bv_{0}^{*}$ in $\vtxg$.
Note that $(\al,\beta)\subset S_{\Gamma}$.
Now there are $\bw_{m},\bv_{m}$ such that $N_{m}(\bv_{0}^{*})\cap(\al,\beta)=[\bw_{m},\bv_{m}]$ for sufficiently large $m$.
Thus, there is $\ut_{m}\in\Gamma$ such that $[\bw_{m},\bv_{m}]\subset\ax(\ut_{m})$.
Note that $[\bw_{i},\bv_{i}]\subset[\bw_{i+1},\bv_{i+1}]$.
Construct a sequence $\{[\bw_{i},\bv_{i}]\}_{i=m}^{\infty}$ from $N_{m}(\bv_{0}^{*})\cap(\al,\beta)$ with $\{\ut_{i}\}_{i=m}^{\infty}\subset\Gamma$ inductively.
By our construction, $[\bw_{i},\bv_{i}]\rightarrow(\al,\beta)$ as $i\rightarrow\infty$.
Since $[\bw_{i},\bv_{i}]\subset\ax(\ut_{i})$ for all $i$, 
we conclude that $\ax(\ut_{i})\rightarrow(\al,\beta)$ as $i\rightarrow\infty$.
\end{proof}

We have two disjoint subsets $\Lambda_{\ax}$ and $\Lambda^{*}$ of $\Lambda$,
$$
	\Lambda=\Lambda_{\ax}\cap\Lambda^{*},
$$
where $\Lambda_{\ax}=\{\gamma^{\pm}:\gamma\in\Gamma\}$, and $\Lambda^{*}=\Lambda-\Lambda_{\ax}$.

\subsection{Construction of examples of a convex cocompact highly-branched Schottky subgroup}
	Recall that any hyperbolic element $\gamma\in G$ is expressed as $\gamma=g_{\gamma}d_{\gamma}g_{\gamma}^{-1}$ for some diagonal element  $d_{\gamma}$ and $g_{\gamma}\in G$,
and that $\ax({\gamma})=(g_{\gamma}.0,g_{\gamma}.\infty)$.
Let us choose 
$$
	g_{\gamma}=\begin{pmatrix}
		\al & \beta \\ 1 & 1
	\end{pmatrix},
$$
so that the axis of $\gamma$ is the geodesic $(\beta,\al)$ with 
$\al\neq\beta$.
Let $\al=\sum_{i=n}^{\infty} \al_{i}p^{i}$ and 
$\al/\mathfrak{p}^{k}=\al\hspace{3pt}{\rm mod}\hspace{3pt}p^{k}$. 
Consider a sequence of vertices
$
	\left\{
	\begin{bmatrix}
		p^{n+k}&\al/\mathfrak{p}^{k}\\0&1
	\end{bmatrix}	
	:k\in\z_{\geq1}
	\right\},
$
that is the geodesic ray from the vertex $\bv_{n}^{*}\in(0,\infty)$ to $\al\in\partial T_{G}$. 
In this sense, we can associate every element $\al\in\partial T_{G}$ to a unique  geodesic ray starting from the vertex $\bv_{\nu_{p}(\al)}^{*}\in(0,\infty)$.

\begin{Ex}\label{ex:subgroup}{\rm
Let $p=3$ and $\omega=\sqrt{2}$.
The quadratic extension $\qp(\sqrt{2})$ of $\qp$ is unramified.
Let $G=\PGLt(\m{Q}_{3}(\sqrt{2}))$, $H=\PGLt(\m{Q}_{3})$, and 
$$
	\gamma_{i}
	=
	g_{i}
	\begin{pmatrix}
		p^{-k_{i}}u_{i} & 0 \\ 0 & 1
	\end{pmatrix}
	g_{i}^{-1}
	=
	\begin{pmatrix}
		p^{-k_{i}}u_{i}a_{i}-b_{i} & (1-p^{-k_{i}}u_{i})a_{i}b_{i} \\
		p^{-k_{i}}u_{i}-1 & -p^{-k_{i}}u_{i}b_{i}+a_{i}
	\end{pmatrix},
$$
where
$
	g_{i}=
	\begin{pmatrix}
		a_{i} & b_{i} \\ 1 & 1
	\end{pmatrix}
	\in G
$
, $k_{i}=1$ or $2$, and $u_{i}\in\mup\xx(1+p\zpo)$.
Let $\varphi$ be a generator of the set of $(p^{2}-1)$-th roots of unity $\mup$.

To satisfy the conditions in Definition \ref{def:acyl}, we have to choose appropriate quadruples $(a_{i},b_{i},k_{i},u_{i})$, and construct sufficiently many $g_{i}$'s. 
Set
{\small
\begin{align*}
	&\qua{1}=(1+\sqrt2\cdot3,1+(1+2\sqrt2)\cdot3,1,1),\\
	&\qua{2}=(1,1+(1+\sqrt{2})\cdot3,1,1),\\
	&\qua{3}=(1+(2+\sqrt{2})\cdot3,1+2\sqrt2\cdot3,1,1),\\
	&\qua{4}=(2+2\sqrt{2},\sqrt{2},1,
	{\bf
	\varphi\cdot
	(\sqrt{1+2\cdot 3^{2}}+\sqrt{2}\cdot 3)
	}),\\
	&\qua{5}=(1+\sqrt{2},1+2\sqrt{2},1,1),\\
	&\qua{6}=(2+\sqrt{2},2,1,1),\\
	&\qua{7}=(1+1\cdot3,0,2,1),\\
	&\qua{8}=(1+2\cdot3,2\sqrt{2},2,1).
\end{align*}
}%

\begin{figure}[t]
  \includegraphics[width=0.8\linewidth]{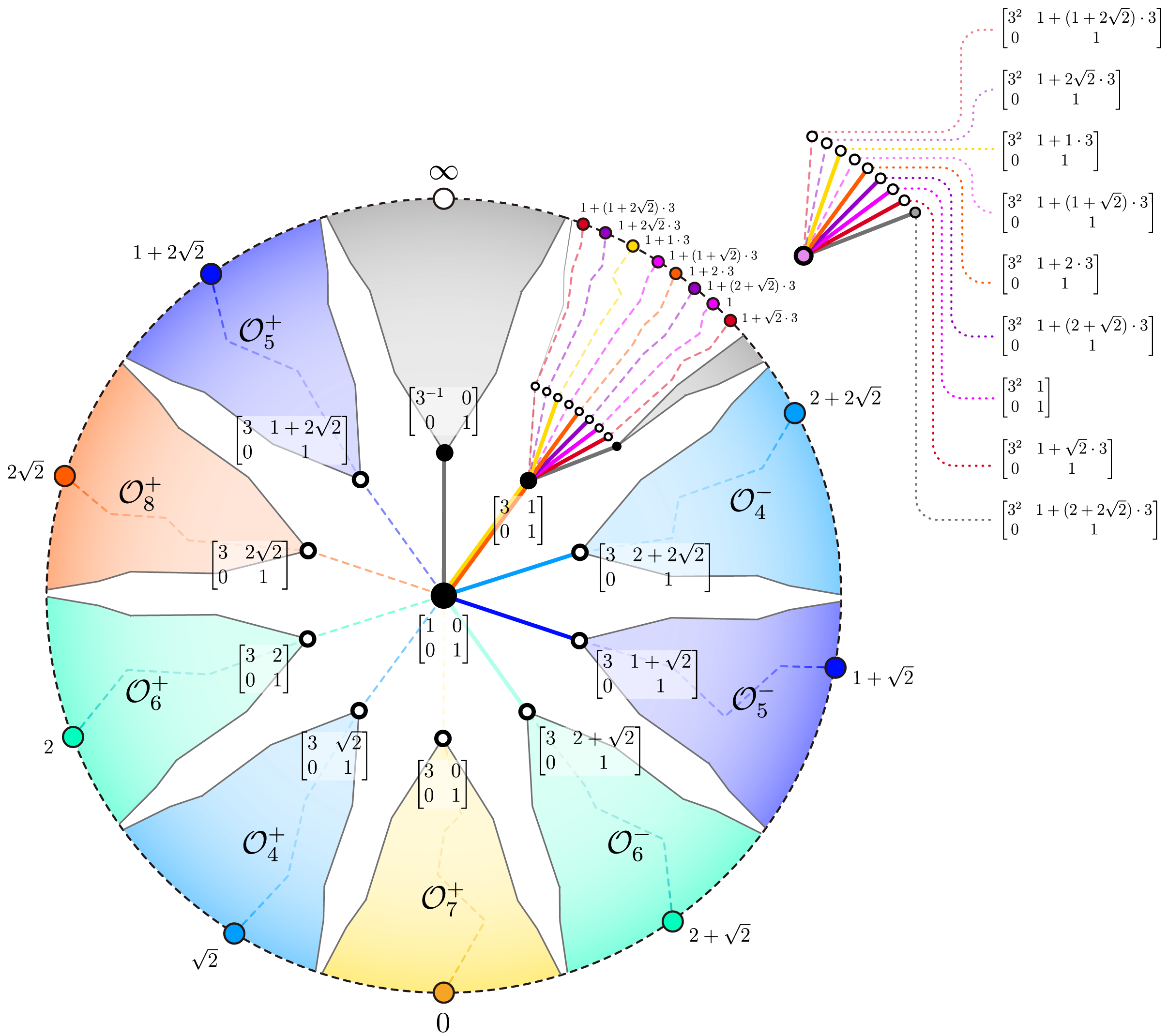}
  \caption{
A Schottky fundamental domain $\cal{F}$ of $\Gamma$ in Example \ref{ex:subgroup}.
The gray area is the non-compact part of $\cal{F}$.
Each color represents a Schottky basis element $\gamma_{i}$.
}
\end{figure}

We emphasized that $u_{4}=\varphi\cdot
	(\sqrt{1+2\cdot 3^{2}}+\sqrt{2}\cdot 3)$. 
Note that
$
	\gamma_{4}
	\simeq
	\begin{pmatrix}
		p^{-1}u_{4} & 0 \\ 0 & 1
	\end{pmatrix},
$
and that $\sqrt{1+2\cdot 3^{2}}+\sqrt{2}\cdot 3\in\scc$ satisfies the condition in Lemma \ref{lem:important}.
Hence, $\langle u_{4}\rangle$ is dense in $\mu_{3^{2}-1}\xx\scc$.

From (\ref{vertexname}), for $a=\sum_{i=m}^{m+n-1}a_{i}p^{i}\in\qpo$ with $n\geq 1$, recall
$$
\bv(a)=
\begin{bmatrix}
	p^{m+n} & a \\ 0 & 1
\end{bmatrix}\in T_{G}.
$$
To construct a Schottky fundamental domain, we label vertices of $\ax(\gamma_{i})=(a_{i},b_{i})$ by
$
\bv_{0}^{i}=\bv(a_{i})
$
letting $\dis\lim_{j\rightarrow+\infty}\bv^{i}_{j}=b_{i}$.
Then $\cal{O}_{i}^{\pm}$ and a Schottky fundamental domain $\cal{F}$ are defined.
Note that the preimage of the convex core of $X$ in $\cal{F}$ is given by 
$$
\cal{F}\cap S_{\Gamma}=
\left\{
	\begin{bmatrix}
		1 & 0\\
		0 & 1
	\end{bmatrix},
	\begin{bmatrix}
		3 & 1\\
		0 & 1
	\end{bmatrix}
\right\}.
$$
Here, 
{\small $
\bv_{1}^{i}=\begin{bmatrix}
		1 & 0\\
		0 & 1
	\end{bmatrix}
$} is contained in
$\ax(\gamma_{i})$ for $i=4,5,6$,
and
{\small$
\bv_{1}^{i}=\begin{bmatrix}
		3 & 1\\
		0 & 1
	\end{bmatrix}
$} is contained in 
$\ax(\gamma_{i})$ for $i=1,2,3$.
Additionally, $\ax(\gamma_{7})$ and $\ax(\gamma_{8})$ contain both
{\small $\begin{bmatrix}
		1 & 0\\
		0 & 1
	\end{bmatrix}$}
and
{\small
$\begin{bmatrix}
		3 & 1\\
		0 & 1
	\end{bmatrix}$}.
	In particular, 
	{\small $\bv_{1}^{7},\bv_{1}^{8}=
	\begin{bmatrix}
		3 & 1\\
		0 & 1
	\end{bmatrix}$}.
In the Schottky fundamental domain $\cal{F}$, the branches at vertices 
{\small $
\begin{bmatrix}
	3^{-1}&0\\0&1
\end{bmatrix}
$}
and
{\small $
\begin{bmatrix}
	3^{2}&1+(2+2\sqrt{2})\cdot3\\0&1
\end{bmatrix}
$}
are preimages of ends in $X$.

\begin{figure}
  \centering
  \begin{subfigure}{0.45\textwidth}
    \includegraphics[width=\textwidth]{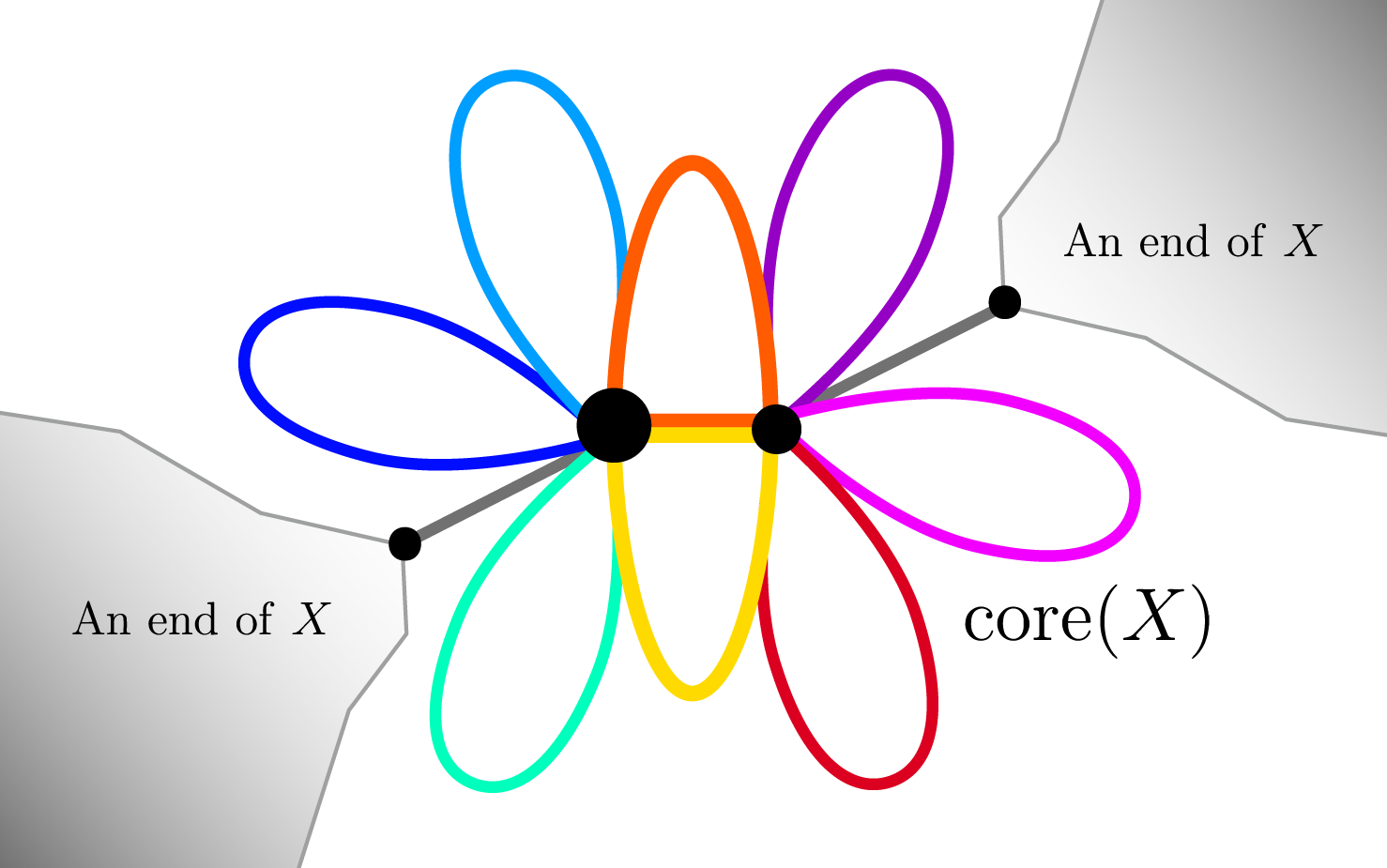}
    \caption{}
  \end{subfigure}
  \hfill
  \begin{subfigure}{0.45\textwidth}
    \includegraphics[width=\textwidth]{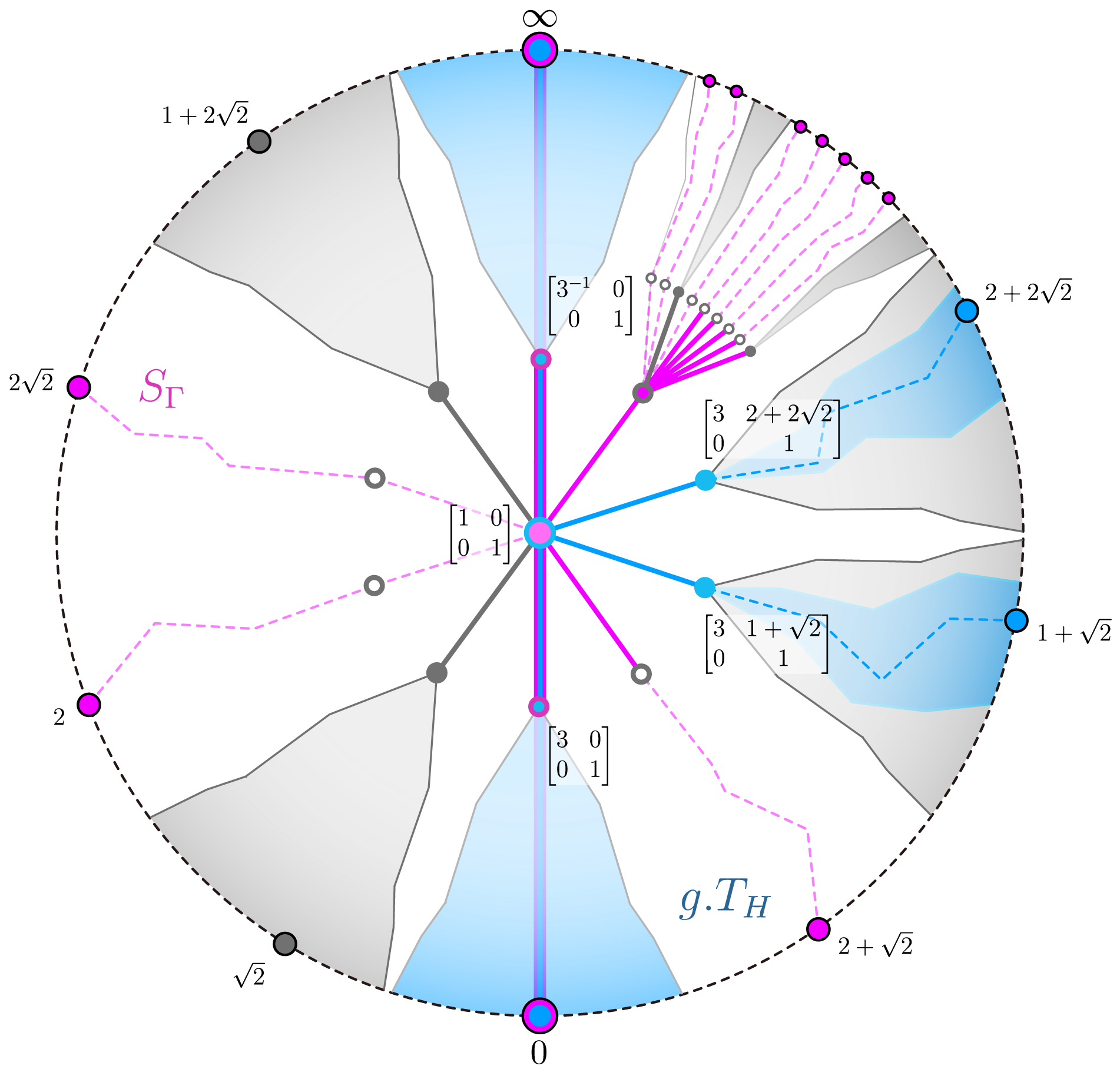}
    \caption{}
  \end{subfigure}
  \caption{
(A) The quotient graph $X$; (B) The counter-example we constructed. 
The blue tree is $g.T_{H}$ and the pink tree is $S_{\Gamma}$.
Here, $\partial(g.T_{H})\cap \partial(S_{\Gamma})=\{0,\infty\}$ by our construction.
 }
  \label{fig:examples}
\end{figure}

In Figure \ref{fig:examples}, we describe $\cal{F}$ and $X$.
Every axis contains a vertex of degree $p^{2}-p+3=9$ in $\cal{F}$.
Therefore, $\Gamma=\langle \gamma_{1},\cdots,\gamma_{8}\rangle$ is an example of a convex cocompact highly-branched Schottky subgroup.
Each dashed line in Figure \ref{fig:examples}(A) is a part of the axis colored by the same color and the invariant subtree $S_{\Gamma}$ contains every colored line.
Meanwhile, Figure \ref{fig:examples}(B) is the quotient graph $X$.
The colored loops form $\core(X)$.
Note that $S_{\Gamma}$ is the universal cover of $\core(X)$ in (B) and that vertices in the invariant subtree $S_{\Gamma}$ have degree 9.

On the other hand, we can construct a non-example with the same method.
Generate $\Gamma'$ by 
$\gamma_{1},\gamma_{2},\gamma_{3},\gamma_{6},\gamma_{8}$, and $\gamma'$, where $\gamma'=\op{diag}(p^{-1},1)$.
Note that $S_{\Gamma'}$ has degree 6 at the origin $\bv_{0}^{*}$.  
At the origin, there are $\sqrt{2},(1+\sqrt{2}),(1+2\sqrt{2}),(2+2\sqrt{2})$ -directions that are not contained in $S_{\Gamma'}$, i.e.,
the branches at
{\small
$
\begin{bmatrix}
3 & \sqrt{2} \\ 0 & 1
\end{bmatrix},
\begin{bmatrix}
3 & 1+\sqrt{2} \\ 0 & 1
\end{bmatrix},
\begin{bmatrix}
3 & 1+2\sqrt{2} \\ 0 & 1
\end{bmatrix},
\begin{bmatrix}
3 & 2+2\sqrt{2} \\ 0 & 1
\end{bmatrix}
$
}
are preimages of ends in $\Gamma'\backslash T_{G}$.

Now, let $g=\op{diag}(1+\sqrt{2},1)\in G$.
We observe that $g.T_{H}$ intersects only $\ax(\gamma')$ among the axes of Schottky basis elements.
It is because $g.T_{H}$ has vertices 
{\small
$
\begin{bmatrix}
3 & 0 \\ 0 & 1
\end{bmatrix},
\begin{bmatrix}
3^{-1} & 0 \\ 0 & 1
\end{bmatrix},
\begin{bmatrix}
3 & 1+\sqrt{2} \\ 0 & 1
\end{bmatrix},
\begin{bmatrix}
3 & 2+2\sqrt{2} \\ 0 & 1
\end{bmatrix}
$}
and the latter two vertices are not contained in $S_{\Gamma'}$.
Furthermore, $\gamma'$ preserves directions at $\bv_{0}^{*}$ along its axis $(0,\infty)$ since $\gamma'$ acts as a pure translation.
It follows that we have $g.\cl{\mathbb{Q}_{3}}\cap\Lambda=\{0,\infty\}$.
Thus, even if we have $g.\cl{\mathbb{Q}_{3}}\cap\Lambda\neq\varnothing$, the intersection is no longer an infinite set.
}
\end{Ex}

\section{Unipotent dynamics for dense orbits}\label{s:unip}

%%%%%%%%%%%%%%%%%%%%%%%%%%%%%%%%%%%%

From now on, we set $\cal{C}=G/H$ under the identification $g_{C}.\cl{\qp}\sim g_{C}H$ as in Section \ref{ss:bt}.
In this section, our goal is to show that if $\Gamma C\subset\cal{C}$ is not closed, then $\cl{\Gamma C}$ contains $\cal{V}$-orbits, where $\cal{V}$ is the $\omega\qp$-unipotent subgroup in $G$.
We will show that any axis of an element of $\Gamma$ is contained in some circle in $\cl{\Gamma C}$.
Results in this chapter will be used to show that $\Gamma C$ is dense in $\cal{C}_{\Lambda}=\{C\in\cal{C}:C\cap\Lambda\neq\varnothing\}$ when $\Gamma C$ is not closed in $\cal{C}$.
We shall start with the non-Archimedean version of the unipotent blowup lemma in \cite{MMO}, which is Theorem \ref{thm:blowup}.

%%%%%%%%%%%%%%%%%%%%%%%%%%%%%%%%%%%%

\subsection{Unipotent blowup and K-thickness}
We first define $K$-thickness which was briefly mentioned in Remark \ref{rmk1}.

\begin{dfn}\label{def:k-thick}
	Let $K\in\N$. 
We call $T\subset \qp$ a $K$-thick set if 
	$$
		(p^{-(K+l)}\zp-p^{-l}\zp)\cap T\neq\emptyset 
	$$
for every $l\in\z$.
\end{dfn}
We remark that if $T\subset\qp$ is $K$-thick, then $-T\subset\qp$ is also $K$-thick.
Recall that 
	$
	u_{t}=
		\begin{pmatrix}
			1&t\\0&1
		\end{pmatrix}
	$
for $t\in\qp$ and that for $a=\sum_{i=n}^{m}a_{i}p^{i}\in\qpo$,
$$
\bv(a)=\begin{bmatrix}
	p^{m+1} & a \\ 0 & 1
\end{bmatrix}.
$$

\begin{prop}\label{prop:k-thick}
	Let $\Gamma$ be convex cocompact highly-branched and $X=\Gamma\backslash T_{G}$.
	There exists $K>1$ such that for all $x\in \RFM$,
	$$
		T(x):=\{t\in\qp : xu_t \in \RFM\}
	$$
	is $K$-thick.
\end{prop}

\begin{proof}
Set $K\geq\op{diam}(\core(X))+1$.
Let $x=\Gamma g$ be an element of $\RFM\subset\Gamma\backslash G$.
Since $\Gamma g \in\RFM$, we have $g.0, g.\infty\in\Lambda$ so that the geodesic $(g.0,g.\infty)$ is in $S_{\Gamma}$.
By high-branchedness of $\Gamma$, there is a sequence $\{\bv_{i_{j}}^{*}\}_{j=-\infty}^{\infty}$ in the geodesic $(0,\infty)$ such that $|i_{j}-i_{j+1}|\leq K$ for every $j$ and
$$
	\left(N_{1}^{g.T_{H}}(g.\bv_{i_{j}}^{*})-\{g.\bv_{{i_{j}}+1}^{*},g.\bv_{i_{j}-1}^{*}\}\right)\cap S_{\Gamma}\neq\varnothing,
$$
where
$
	N_{1}^{g.T_{H}}(\bv)=\{\bw\in{\bf{V}}(g.T_{H}):d(\bv,\bw)=1\}
$
is the 1-neighborhood of $\bv$ in $g.T_{H}$.
Thus, there exists some $t_{j}\in(\zp/p\zp)^{\xx}$ such that
{
$$
	g.\bv(t_{j}p^{i_{j}})\in
	\left(N_{1}^{g.T_{H}}(g.\bv_{i_{j}}^{*})-\{g.\bv_{i_{j}+1}^{*},g.\bv_{i_{j}-1}^{*}\}\right)\cap
	S_{\Gamma}.
$$
}
Since {$g.\bv(t_{j}p^{i_{j}})\in S_{\Gamma}$}, by high-branchedness again, we have $t_{j+1}\in(\zp/p\zp)^{\xx}$ such that
{
$$
	g.\bv(t_{j}p^{j}+t_{j+1}p^{i_{j}+1})\in
	N_{1}^{g.T_{H}}
	\left(
	g.\bv(t_{j}p^{j})
	\right)
	\cap
	S_{\Gamma}.
$$
}
We inductively have $t_{j+N-1}\in(\zp/p\zp)^{\xx}$ such that 
{
$$
	g.\bv\left(\sum_{n=0}^{N-1}t_{j+n}p^{i_{j}+{n}}\right)\in
	N_{1}^{g.T_{H}}
	\left(
	g.\bv\left(\sum_{n=0}^{N-2}t_{j+n}p^{i_{j}+{n}}\right)
	\right)
	\cap
	S_{\Gamma}.
$$
}
Define
$
	\al_{i_{j}}=\sum_{n=0}^{\infty}t_{j+n}p^{i_{j}+n}
$
so that $\al_{i_{j}}\in\qp^{\xx}$.
We just constructed a geodesic ray $g.[\bv_{i_{j}}^{*},\al_{i_{j}})$ containing the vertices $g.\bv\left(\sum_{n=0}^{N-1}t_{j+n}p^{i_{j}+{n}}\right)$'s, which converge to $g.\al_{i_{j}}.$
Consequently, $gu_{\al_{i_{j}}}.0=g.\al_{i_{j}}\in\partial S_{\Gamma}=\Lambda$.
It is obvious that
$gu_{\al_{i_{j}}}.\infty=g.\infty$, thus $\Gamma gu_{\al_{i_{j}}}\in\RFM$.
Furthermore, we have $\pn{\al_{i_{j}}}=p^{-i_{j}}$ and $|i_{j}-i_{j+1}|\leq K$.
Thus,
$$
	T'(x):=\{\al_{i_{j}}\in\qp:\Gamma gu_{\al_{i_{j}}}\in\RFM, j\in\z\}
$$
is $K$-thick.
Since $T'(x)$ is a subset of $T(x)$, the set $T(x)$ is also $K$-thick.
\end{proof}

We remark that a converging sequence of $K$-thick sets for a fixed $K$ converges to a $K$-thick set. 
We will follow Section 8 of \cite{MMO} for the following lemma.
However, we will use ultrametric properties of non-Archimedean local fields, which enable us to obtain the exact constant $p^{-Kd}$ in the lower bound in the following lemma.

\begin{lemma}\label{lem:poly_bound}
Let $K>1$, $d\in\N$ and let $T$ be a $K$-thick set. 
Fix a polynomial $f\in\zpo[t]$ of degree $d$. 
For any large enough ball $B_{m}=p^{-m}\zp$ with $m\gg0$,
the following inequality holds:
$$
	p^{-Kd}\cdot\max_{t\in B_{m}}\pn{f(t)}\leq\max_{t\in T\cap B_{m}}\pn{f(t)}.
$$
\end{lemma}

\begin{proof}
Let 
$
	f(t)=\al_{d}t^{d}+\al_{d-1}t^{d-1}+\cdots+\al_{1}t+\al_{0},
$
where $\al_{i}\in\zpo$ with $\nu_{p}(\al_{i})=n_{i}$ for some $n_{i}\in\z$.
Since $\qpo$ is ultrametric, we have
\begin{align*}
	\pn{f(t)}=\max_{i=0,\cdots,d}\left\{\pn{\al_{d-i}}\pn{t}^{d-i}\right\}=\pn{\al_{d}}\pn{t}^{d}
\end{align*}
if $\pn{t}$ is large enough compare to $\pn{\al_{i}}$'s.
Thus, the maximum of $\pn{f(t)}$ over the set $B_{m}$ is attained at 
{\Small$\exists$\hspace{1.5pt}}$t_{0}\in p^{-m}\zp^{\xx}$ for every large enough $m$.
It follows that $\pn{f(t_{0})}=p^{-n_{d}+dm}$.

Now let $B'_{m}=p^{-m}\zp-p^{-m+K}\zp$.
By $K$-thickness of $T$, we have $T\cap B'_{m}\neq\varnothing$.
Note that $\dis\max_{t\in T\cap B_{m}}\pn{f(t)}=\max_{t\in T\cap B'_{m}}\pn{f(t)}$.
Choose $t_{1}\in T\cap B'_{m}$ such that
$$\dis\max_{t\in T\cap B'_{m}}\pn{f(t)}=\pn{f(t_{1})}.$$
Let $n_{t_{1}}:=\nu_{p}(t_{1})\in[-m,-m+K]$ so that
$
	\pn{f(t_{1})}=p^{-n_{d-s}-(d-s)n_{t_{1}}}
$
for some $s\in\{0,1,\cdots,d\}$.
By the maximality of $\pn{f(t_{0})}$ and $\pn{f(t_{1})}$, we obtain
\begin{align*}
	\max_{t\in T\cap B_{m}}\pn{f(t)}
	&=
	p^{-n_{d-s}-(d-s)n_{t_{1}}}
	\geq
	p^{-n_{d}-dn_{t_{1}}}
	\geq
	p^{-n_{d}-d(-m+K)}\\
	&=p^{-dK}\cdot \pn{f(t_{0})}=p^{-dK}\cdot\max_{t\in B_{m}}\pn{f(t)}.
\end{align*}
\end{proof}

%%%%%%%%%%%%%%%%%%%%%%%%%%%%%%%%%%%%
	In the proof of the next theorem, we use the exponential map for Lie algebra over non-Archimedean local fields.

\noindent
\begin{rmk}\label{rmk3}
{\rm
Let $G':=\GLt(\qpo)$. 
The exponential map, 
$$
\exp:B_{p^{r}}(0)\subset\Lie(G')\longrightarrow B_{p^{r}}(e)\subset G',
\hspace{10pt}
\exp(\rho)=\sum_{n=0}^{\infty}\frac{\rho^{n}}{n!}
$$
is well-defined when $r\leq-1/(p-1)$.
In our case, as $\qpo$ is unramified, it is enough to set $r=-1$.
By ultrametric property, $\norme{\exp(\rho)}=\norm{\exp(\rho)-e}$ depends only on the second term in the summation, which is $\rho$ (see Section 5.4 in \cite{Ro}).
In particular, $\exp$ is an isometry on $B_{p^{r'}}(0)$ (see \cite{Bour}, Section 7 of Chapter 3, and \cite{Sch}, Section 31, 32).
%When we use the Lie algebra structure, we will work on $G'$ and then projectivize the results to $G$.
%We will consider the subgroups $\cal{A},\cal{V},\cal{U}$, and $\cal{N}$ as subgroups of $G'$ (see (\ref{matrices}) for definition).
}
\end{rmk}
%%%%%%%%%%%%%%%%%%%%%%%%%%%%%%%%%%%%

The proofs of the following two Theorems are as in \cite{MMO}.
One needs to use $p$-adic Lie theory with Lemma \ref{lem:poly_bound} to follow their proof.
See also the detailed version \cite{JL} for details.

\begin{theorem}\label{thm:blowup}
Suppose that $g_n\rightarrow e$ in $G-\cal{AN}$ and $g_{n}'\rightarrow e$ in $G-\cal{A}\cal{N}$.
Then,
\begin{enumerate}
	\item There are $u_{t_{n}}\in \cal{U}$ and $h_n\in H$ such that after passing to a subsequence,
$u_{t_{n}}g_nh_n\rightarrow g\in \cal{V}-\{e\}.$
	\item There are $u_{t_{n}},u_{s_{n}}\rightarrow\infty$ in $\cal{U}$ such that after passing to a subsequence, $u_{t_{n}}g_nu_{s_{n}}\rightarrow g\in \cal{AV}-\{e\}.$

\end{enumerate}
Furthermore, given a sequence of $K$-thick sets $T_n\subset\qp$ and a neighborhood $G_0\subseteq B_{p^{-1}}(e)$ of $e$ in $G$, we can arrange so that $t_{n}\in T_n$ and $g\in G_0$ in both case.
\end{theorem}

\subsection{Minimality}
Recall that
$$
	\cal{U}=\left\{u_{a}=\begin{pmatrix}
	1&a\\0&1
	\end{pmatrix}: a\in\qp \right\},\ 
	\cal{V}=\left\{u_{\omg a}=\begin{pmatrix}
	1&\omega a\\0&1
	\end{pmatrix}: a\in\qp \right\}.
$$
\begin{dfn}\label{def:min}
Let $W$ be a closed subgroup of $G$. Let $Y\subset\FM$ be non-empty and $Y^*=Y\cap\RFM$. 
Then, we say that
\begin{enumerate}
    \item $Y$ is \emph{W-minimal} if $\cl{yW}=Y$ for any $y\in Y$.
    \item $Y$ is \emph{relatively W-minimal} if $Y^*$ is non-empty, and $\cl{yW}=Y$ for all $y\in Y^*$.
\end{enumerate}
\end{dfn}

\begin{prop}\label{prop:meet_RFM}
Every closed $W$-invariant set intersecting $\RFM$ contains a relatively $W$-minimal set.
\end{prop}
This is a consequence of Zorn's lemma.
Before we state Proposition \ref{prop:AV density}, we define the following.
\begin{dfn}	
We say that a subset $S$ is a \emph{$\zp$-parameter semigroup of $\GLt(\qpo)$} 
if there is a vector $\zeta\in\Lie(\GLt(\qpo))$ such that 
$$
	S=\{\exp(\alpha\zeta):\alpha\in\zp\}.
$$
Call $S$ a $\zp$-parameter semigroup of $G$ if it is the image of a $\zp$-parameter semigroup of $\GLt(\qpo)$ under the projection to $G$.
\end{dfn}

%%%%%%%%%%%%%%%%%%%%%%%%%%%%%%%%%%%%
Now we define the following subsets of $\cal{V}$ and $\cal{A}$.
For fixed $\vep=p^{-k}$ for some $k\in\N$, let
\begin{align*}
	\cal{V}_\vep=\left\{u_{b}: b\in\qpo,\pn{b}\leq\vep\right\}, 
	\hspace{5pt}
	\cal{A}_\vep=\{\diag(a,1): a\in 1+p^k\zp\}.
\end{align*}

\begin{prop} \label{prop:AV density}
For any relatively $\cal{U}$-minimal set $Y\subset\FM$, there is $L\subset G$ such that $L$ is dense in one of $\zp$-parameter semigroup 
$\cal{V}_{\vep}, \cal{A}_{\vep}, v\cal{A}_{\vep}v^{-1}$ for some $v\in\cal{V}-\{e\}$ for some $\vep\leq p^{-1}$, 
and $YL\subset Y$.
\end{prop}
 
\begin{proof}
Let $Y$ be a relative $\cal{U}$-minimal set. 
Recall that $Y^{*}=Y\cap\RFM$. 
For Step 1 and Step 2, the proofs are as in \cite[Theorem 9.4 \& Lemma 9.5]{MMO}. 
See also the detailed version \cite{JL}.
Step 3 is also similar, however, there is a subtle difference due to $p$-adic Lie theory.

\noindent{\bf{Step 1.}}
For any $y\in Y^*$, there is a sequence $g_n\rightarrow e$ in $G-\cal{U}$ such that $yg_n\in Y$ for all $n\in\N$. 

\noindent{\bf{Step 2.}}
There is a sequence $q_{n}\rightarrow e$ in $\cal{A}\cal{V}$ such that $Yq_{n}\subset Y$.

Among $\{q_{n}\}$ obtained by Step 2, choose $q$
such that $\norme{q}\leq p^{-1}$. 
Regard $q$ as an element of $\GLt(\qpo)$.
By Remark \ref{rmk3}, there exists a vector $\zeta\in\Lie(\GLt(\qpo))\cap B_{p^{-1}}(0)$ such that $\exp(\zeta)=q$ in $\GLt(\qpo)$.

We remark that if $Yq\subset Y$, then $Yq^{n}\subset Y$ for any $n\in\N$. 

\vspace{0.1in}

\noindent{\bf{Step 3.}} The element $q\in\cal{A}\cal{V}$ defined above is an element of one of 
$\cal{V}_{\vep}$,
$\cal{A}_{\vep}$, and
$v\cal{A}_{\vep}v^{-1}$
for some $v\in \cal{V}-\{e\}$ and $\vep=\norme{q}$.
\begin{proof}[Proof of Step 3.]
Since $q=\exp(\zeta)$ is an element of $\cal{A}\cal{V}$, we have
$$
\zeta=
	\begin{pmatrix}
		a&b\\0&d
	\end{pmatrix}
\in\Lie(\GLt(\qpo)),
$$
where $b\in\omega(p\zp)$ and $\norm{\zeta}=\max\{\pn{a},\pn{b},\pn{d}\}\leq p^{-1}$. 
Let $\vep=\pn{a}$ if $a\neq0$, and
$\vep=\pn{b}$ if $a=0$.
Since we deal with the projective linear group, we may let $d=0$.

When $a=0$, one can show that $\exp(\zeta)\in\cal{V}_{\vep}$.
Similarly, when $b=0$, we have $\exp(\zeta)\in\cal{A}_{\vep}$.
Lastly, if $a,b\neq0$, then note that
$
	\zeta^{n}=
	\begin{pmatrix}
	a^{n}&a^{n-1}b\\0&0
	\end{pmatrix}
$
for $n\geq 1$. 
We then have
$
	\exp(\zeta)
	=v\cdot\op{diag}(\exp(a),1)\cdot v^{-1},
$
where 
$
v=u_{-a^{-1}b}\in\cal{V}
$
Hence, $\exp(\zeta)\in v\cal{A}_{\vep}v^{-1}$.

\end{proof}

To finish the proof, let $L$ be $\langle q\rangle=\{q^{n}:n\in\N\}$. 
Note that $q^{n}=(\exp(\zeta))^{n}=\exp(n\zeta)$ for every $n\in\N$.
Since $\N$ is dense in $\zp$, the set $\{n\zeta:n\in\N\}$ is dense in 
$\zp\zeta$. 
Note that $\zp\zeta$ is either
$
	\left\{
	\begin{pmatrix}
		0&\beta\\0&0
	\end{pmatrix}
	: \pn{\beta}\leq\vep
	\right\}
$
with $\vep=\pn{b}$, or
$
	\left\{
	\begin{pmatrix}
		\al&0\\0&0
	\end{pmatrix}
	: \pn{\al}\leq\vep
	\right\},
	\left\{
	\begin{pmatrix}
		\al a& \al b \\0&0
	\end{pmatrix}
	: \pn{\al}\leq\vep
	\right\}
$
with $\vep=\pn{a}$.
Hence, $\exp(\zp\zeta)$ is one of $\cal{V}_{\vep},\cal{A}_{\vep},v\cal{A}_{\vep}v^{-1}$ by Claim 3.
Thus, $L$ is dense in the one of the $\zp$-parameter semigroups $\cal{V}_{\vep}, \cal{A}_{\vep}, v\cal{A}_{\vep}v^{-1}$ for some $v\in\cal{V}$ and $YL\subset Y$ by construction.
Thus Proposition follows.
\end{proof}

\begin{rmk}{\rm
Note that the domain of $\exp$ is limited to $B_{r}(0)\in\zpo$ for $r<1$ in our case. 
The unipotent subgroups $\cal{U}$ and $\cal{V}$ are the only \emph{1-parameter subgroups} of $G$ (see Section 2.1 of \cite{BQ}).
We have defined \emph{$\zp$-parameter semigroup} to avoid such difficulty.
}
\end{rmk}

\begin{dfn}
Let $z\in\FM$ such that $zH$ is not closed and $zH\cap\RFM\neq\varnothing$.
We say that such an orbit $zH$ is chaotic. 
\end{dfn}

	Our goal is to analyze $\cl{zH}$ when $zH$ is chaotic.
By Proposition \ref{prop:meet_RFM}, there is a relatively $\cal{U}$-minimal set $Y\subset\cl{zH}$.
Applying Proposition \ref{prop:AV density} to $Y$, we obtain  $L\subset G$ such that $YL\subset Y$ and is dense in one of $\cal{V}_{\vep}, \cal{A}_{\vep}, v\cal{A}_{\vep}v^{-1}$ for some $v\in\cal{V}-\{e\}$ for some $\vep\leq p^{-1}$.
If $L$ is dense in either $\cal{V}_\vep$ or $v\cal{A}_{\vep}v^{-1}$,  then $L$ has elements with (1,2)-entry in $\qpo-\qp$.

%%%%%%%%%%%%%%%%%%%%%%%%%%%%%%%%%%%%

%%%%%%%%%%%%%%%%%%%%%%%%%%%%%%%%%%%%
Theorem \ref{thm:transl} is used to prove Corollary \ref{cor:dense}.
We will use Corollary \ref{cor:dense} to prove Proposition \ref{prop:goodcircle}.
The proof of Theorem \ref{thm:transl} requires Proposition \ref{prop:AV density} and it is as in \cite[Theorem 9.1]{MMO}.

%\begin{lemma} \label{lem:translGH}
%Let $z\in\FM$ such that $zH$ is chaotic. 
%For any $q\in Z=\cl{zH}$, there exists $g_{n}\rightarrow e$ in $G-H$ such that $qg_{n}\in Z$ for all $n$.
%
%Moreover, for a relatively $\cal{U}$-minimal set $Y\subset Z$,
%there exists a non-identity element $w$ of $\cal{V}$ such that $Yw\subset Z$.
%
%\end{lemma}

\begin{thm}\label{thm:transl}
Let $zH$ be chaotic and let $Z$ be its closure.
There exist $y\in Z$ and a subset $L\subset G$ such that $yL\subset Z$ and $\cl{L}$ is a $\zp$-parameter semigroup intersecting $H$ trivially.
\end{thm}

%%%%%%%%%%%%%%%%%%%%%%%%%%%%%%%%%

\begin{rmk}
{\rm
Let $\vep=p^{-k}$.
By Theorem \ref{thm:transl}, $\cl{L}$ is $\cal{V}_{\vep}$ or $v\cal{A}_{\vep}v^{-1}$ for some $v\in\cal{V}$.
Denote the set of the (1,2)-entries of elements of $L$ by $B_{1,2}$.
When $\cl{L}=\cal{V}_{\vep}$, we know that $B_{1,2}$ is dense in $\omg p^{k}\zp$.
When $\cl{L}=v\cal{A}_\vep v^{-1}$, we remark that $B_{1,2}$ is dense in $\omg p^{k'}\zp$ for some $k'>k$.

Let
	$
		v=u_{\omg b}	
	$
for some $b\in\qp$.
Let $d_{a}=\operatorname{diag}(a,1)\in\cal{A}_\vep\cap v^{-1}Lv$, where $a\in 1+p^{n_{a}}\zp^{\xx}$ for some $n_{a}\geq k$.
We have,
	$
		vd_{a}v^{-1}=\begin{pmatrix} a&\omega b(1-a)\\0&1 \end{pmatrix}.
	$
Hence if we let $b=p^{n_b}\beta$ for some $\beta\in\zp^\times$, then 
$
	b(a-1)=p^{n_b+n_{a}}c
$
for some $c\in\zp^\times$. 
Since $v^{-1}\cl{L}v=\cal{A}_{\vep}$,
for any $n\geq k$, there exists $s\in 1+p^{n}\zp^{\xx}$ such that $d_{s}\in\cal{A}_\vep\cap v^{-1}Lv$.
By varying $n\geq k$,
we have $B_{1,2}$ is dense in $\omega p^{n_{b}+n_{a}}\zp\subset\omega\zp$.
}
\end{rmk}
%%%%%%%%%%%%%%%%%%%%%%%%%%%%%%%%%

\begin{Cor}\label{cor:dense}
Let $zH$ be a chaotic orbit and $Z=\cl{zH}$.
	There exists $L'\subset G$ such that $L'$ is dense in $\cal{V}$ and $z'L'\subset Z$ for some $z'\in\RFM$, consequently, $z'\cal{V}\subset Z$.
\end{Cor}

\begin{proof}
Let a $\omg\qp$-unipotent matrix $u_{\omg b}$ be as in (\ref{matrices}), and $L$ be as in Theorem \ref{thm:transl}.

\noindent{\bf Case 1.} $L$ is dense in $\cal{V}_{\vep}$ for some $\vep=p^{-k}$: 
Let $d=\op{diag}(p^{-1},1)$.
For every $n\in\z,b\in p^{k}\zp$,
$
	d^{n}u_{\omg b}d^{-n}=
	\begin{pmatrix}
		1 & \omega (p^{-n}b) \\ 0 & 1
	\end{pmatrix}.
$
Let $L'=\bigcup_{n\in\z} d^{n}Ld^{-n}$ so that $L'$ is dense in $\cal{V}$.
Let $y$ be as in Theorem \ref{thm:transl}.
Note that $yd^{-n}\in Z$ for all $n\in\z$ since $Z$ is right $H$-invariant.
Moreover, $yd^{-n}\rightarrow z'$ for some $z'\in\RFM$, by passing to a subsequence if necessary, since $\RFM$ is compact $\cal{A}$-invariant.
Hence,
$$
	yLd^{-n}=yd^{-n}\cdot d^{n}Ld^{-n}\subset Z=\cl{zH},
$$
which gives $z'L'\subset Z$ as $n\rightarrow\infty$.

\noindent{\bf Case 2.} $L$ is dense in $u_{\omg b}\cal{A}_{\vep}u_{\omg b}^{-1}$ for some $\vep=p^{-k}$:
Let $d_{a}=\op{diag}(a,1)$, where $a\in(1+p^{k}\zp)$.
We obtain $u_{\omg b}d_{a}u_{\omg b}^{-1}=v_{-b(a-1)}d_{a}$.
Since $Z$ is right $H$-invariant, $v_{-b(a-1)}\in LH$.
Hence, there is $L''\subset G$ such that $L''$ is dense in $\cal{V}_{\vep'}$, where $\vep'=p^{-\nu_{p}(b)-k}$.
Thus, it is reduced to Case 1 by setting $L=L''$, which leads to the conclusion.

%{\ce
%\noindent \emph{Proof of (2).} Choose an arbitrary element $v\in\cal{V}$ and consider $xv\in x\cal{V}$.
%There is a sequence $\{v_{i}\}_{i=0}^{\infty}\subset L'$ such that 
%$v_{i}\rightarrow v$ as $i\rightarrow \infty$
%by Corollary \ref{cor:dense}.
%Thus, $xv_{i}\rightarrow xv$ as $i\rightarrow\infty$.
%Since each $xv_{i}$ is an element of $Z=\cl{zH}$, there are sequences such that
%$
%	zh_{i,j}\rightarrow xv_{i}
%$
%as $j\rightarrow\infty$.
%Define the diagonal sequence by $\{zh_{k,k+1}\}_{k=0}^{\infty}$.
%We then have $zh_{k,k+1}\rightarrow xv$ as $k\rightarrow\infty$.
%Since $Z$ is closed, we conclude that $xv\in Z$.}
\end{proof}

\subsection{Axis denseness of $\Gamma$-orbit}
We identified $C$ with $\partial(\hull(C))=\partial(g_C.T_H)$.
From Section \ref{ss:bt},
recall that $N_n(\bv)$ the $n$-neighborhood of $\bv$ in $T_G$ 
and that,
for two adjacent vertices $\bv,\bw\in T_{G}$, $B_{\bv}^{\bw}$ is a branch at $\bv$ in $\bw$-direction.

\begin{prop}\label{prop:infsubset}
If $C\cap \Lambda\neq\varnothing$, then it is an infinite subset of $\Lambda$.
\end{prop}
\begin{proof}
Choose $\alpha$ in $C\cap\Lambda$. 
We then have a geodesic ray $[\bv,\al)\subseteq\hull(C)\cap S_{\Gamma}$. 
Now let
$$
N^{S_{\Gamma}}_{n}(\bv)=N_{n}(\bv)\cap S_{\Gamma.}
$$
	Condition (1) in Definition \ref{def:acyl} ensures that
if we let $r\geq\op{diam}({\core(X)})+1$, then there is 
	$\bw_{1,0}\in N^{S_{\Gamma}}_r(\bv)\cap[\bv,\al)$
such that $\op{deg}_{S_{\Gamma}}(\bw_{1,0})\geq p^2-p+3$.
Remark that $\hull(C)$ is a regular subtree of degree $p+1$. 
Hence, there is a vertex $\bw$ which is adjacent to $\bw_{1,0}$ and $[\bw_{1,0},\bw]\subset(\hull(C)\cap S_{\Gamma})-[\bv,\al)$.
Let 
$$
	\nrc{\bv}=N_{r}^{S_{\Gamma}}(\bv)\cap\hull(C).
$$
There exists $\bw_{1,1}\in \nrc{\bw_{1,0}}\cap B_{\bw_{1,0}}^{\bw}$ 
such that $\op{deg}_{S_{\Gamma}}(\bw_{1,1})\geq p^{2}-p+3$ and 
$d(\bv,\bw_{1,0})<d(\bv,\bw_{1,1})$.
Inductively, we can find $\bw_{1,j}\in\nrc{\bw_{1,j-1}}$ for $j\in\z_{>0}$ 
such that $\op{deg}_{S_{\Gamma}}(\bw_{1,j-1})\geq p^{2}-p+3$ and 
$d(\bv,\bw_{1,j-1})<d(\bv,\bw_{1,j})$.
Eventually, $\bw_{1,j}$ converges to some $\beta_{1}\in \Lambda-\{\al\}$ as $j\rightarrow\infty$.
Remark that $[\bw_{1,0},\beta_{1})\subset\hull(C)$ by our construction. 

In the same way, we can construct $\{\bw_{i,j}\}_{j=0}^{\infty}$ for all $i\geq 1$ and their limits $\{\beta_{i}\}_{i=1}^{\infty}\subset\Lambda$.
Hence, $\{\beta_{i}\}_{i=1}^{\infty}\subset C\cap\Lambda$ so that $|C\cap\Lambda|=\infty$.
\end{proof}

\begin{figure}[h!]
  \includegraphics[width=0.3\linewidth]{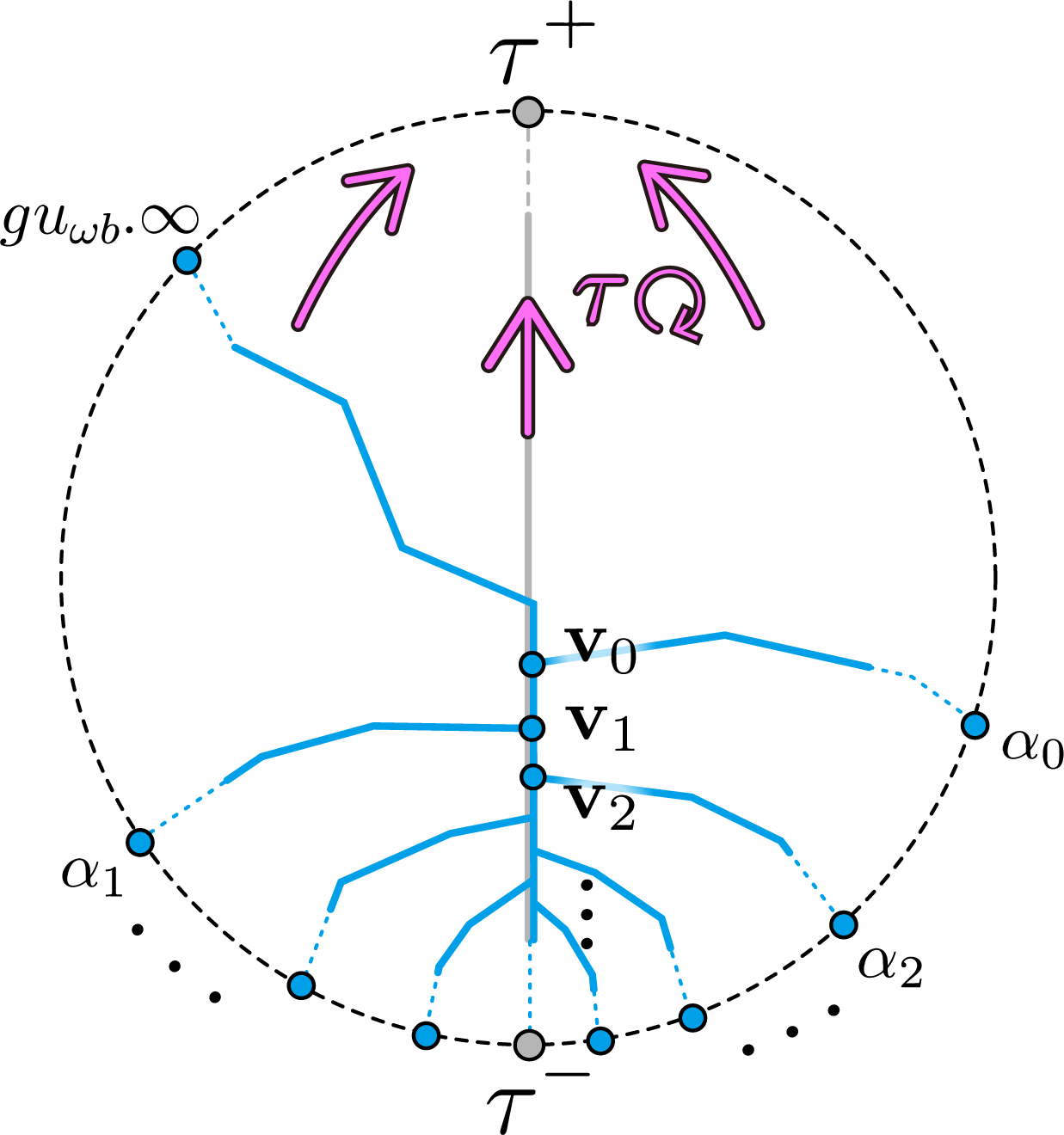}
  \caption{Simplified description of situation in the proof of Proposition \ref{prop:goodcircle}. Blue lines are parts of $gu_{\omg b}.T_{H}$.}
  \label{fig:tt}
\end{figure}

\noindent
\begin{rmk}\label{rmk4}
{\rm
Let $C\in\cal{C}$ be given and set $C=gH=g.\cl{\qp}$ for some $g\in G$.
We remark that for any given three mutually distinct points $g.\al_{1},g.\al_{2},\beta\in g.\cl{\qp}$, we can find a representative $g_{C}$ of $C$ as an $H$-coset such that
\begin{equation}\label{eqn:coor}
	(g_{C}.0,g_{C}.\infty,g_{C}.1)=(g.\al_{1},g.\al_{2},\beta).
\end{equation}
}
\end{rmk}

\begin{Prop}\label{prop:goodcircle}
Suppose that $\Gamma C$ is not closed and $C\cap\Lambda$ is non-empty. Then for every $\ut\in\Gamma$, there is a circle $D\in\cl{\Gamma C}$ such that $\ut^{+},\ut^{-}\in D$.
\end{Prop}

\begin{proof}
Let $\ut=g_{\ut}d_{\ut}g_{\ut}^{-1}\in\Gamma$ be a hyperbolic element, 
where $g_{\ut}\in G$ and $d_{\ut}=\op{diag}(p^{-n_{\ut}}a_{\ut},1)$ for some $n_{\ut}\in\N$ and $a_{\ut}\in\zpo^{\xx}$.
Recall that $n_{\ut}$ is the translation length of $\ut$ and that $g_{\ut}.0=\ut^{-}$ and $g_{\ut}.\infty=\ut^{+}$.

Let $z=\Gamma g_{C}\in\FM$.
We identified $\Gamma\backslash\cal{C}$ and $\FM/H$ in Section \ref{ss:bt}.
Indeed, $\Gamma C\neq\cl{\Gamma C}$ implies that $zH\neq\cl{zH}$.
It follows that ${zH}$ is a chaotic orbit.
By Corollary \ref{cor:dense}, there exists $g\in G$ such that $\Gamma g\cal{V}\subset \cl{zH}=\cl{\Gamma g_{C}H}=\cl{\Gamma C}$.
Let $D'=gH$ so that $D'\in\cl{\Gamma C}$.
Note that $D'$ intersects $\Lambda$ so that $D'\cap\Lambda$ is an infinite set by Proposition \ref{prop:infsubset}.
By Remark \ref{rmk4}, we may assume that $g.\infty\in\Lambda$.

Now consider $u_{a+\omg b}\in\cal{N}$ as in \ref{matrices} for $a,b\in\qp$.
We have $u_{a+\omg b}.0=a+\omega b$.
Thus we have $gu_{a+\omg b} H=gu_{\omg b}u_{a}H=gu_{\omg b}H$ and we can use  $gu_{\omg b}$ instead of $gu_{a+\omg b}$ to translate the point $0\in\cl{\qpo}$ to $\ut^{-}$.
Hence, we may assume that $gu_{\omg b}.0=\ut^{-}$.
Let $D''=gu_{\omg b}H$.
Here we remark that $D''\in\cl{\Gamma C}$, $\ut^{-}\in gu_{\omg b}.\cl{\qp}$, and $gu_{\omg b}.\infty=g.\infty$.

Since two geodesics $(\ut^{-},\ut^{+})$ and $(\ut^{-},gu_{\omg b}.\infty)$ share a common end point,
$
	(\ut^{-},\ut^{+})\cap(\ut^{-},gu_{\omg b}.\infty)
$
is an infinite set in $\vtxg$.
In particular, the intersection $(\ut^{-},\ut^{+})\cap(\ut^{-},\beta)$ for any 
$\beta\in (gu_{\omg b}.\cl{\qp}-\{\ut^{-}\})$ is a geodesic ray.
Recall that $(\ut^{-},\ut^{+})=\ax(\ut)$.
Let $\{\bv_{i}\}_{i=0}^{\infty}$ be a sequence in $\ax(\ut)$ towards $\tau^{-}$ such that 
$[\bv_{0},\ut^{+})\cap[\bv_{0},gu_{\omg b}.1)=\{\bv_{0}\}$
and $d(\bv_{0},\bv_{i+1})-d(\bv_{0},\bv_{i})=1$ for all $i\in\z_{\geq1}$.
(see Figure \ref{fig:tt}).

On the other hand, there is a sequence $\{\al_{i}\}_{i=0}^{\infty}$ in $gu_{\omg b}.\cl{\qp}$ such that $(\ut^{-},\al_{i})\cap\ax(\ut)=(\ut^{-},\bv_{i}]$.
Note that $\ut.(\ut^{-},\al_{i})=(\ut^{-},\ut.\al_{i})$ and
$$
	(\ut^{-},\ut.\al_{i})\cap\ax(\ut)=(\ut^{-},\bv_{i-n_{\ut}}]
$$
for $i\geq n_{\ut}$ since $\ut$ translates vertices along $\ax(\ut)$.
Now consider a sequence of frames, $g_{0}=gu_{\omg b}$ and
$
	g_{j}\sim(\ut^{-},gu_{\omg b}.\infty,\al_{jn_{\ut}})
$
for all $i\geq 1$.
Note that $g_{j}$'s are all representatives of $D''=gu_{\omg b}H$ since each triple consists of three mutually distinct elements of $gu_{\omg b}.\cl{\qp}$.
Hence if we consider $\cl{\Gamma C}$ as a subset of $G$, then  $g_{j}\in\cl{\Gamma C}$ for all $j$.
Applying $\ut$-action and passing to a subsequence if necessary, we have 
$\ut^{j}.(\ut^{-},gu_{\omg b}.\infty,\al_{jn_{\ut}})
	\rightarrow(\ut^{-},\ut^{+},\al)
$
as $j\rightarrow\infty$
for some $\al\in\cl{\qpo}$ such that $\ax(\ut)\cap(\ut^{-},\al)=(\ut^{-},\bv_{0}]$.

By our construction, $\ut^{-},\ut^{+},\al$ are mutually distinct.
Thus, let $g^{*}$ be the element of $G$ corresponding to the frame $(\ut^{-},\ut^{+},\al)$.
Let $D=g^{*}H$.
We conclude that $D\in\cl{\Gamma C}\subset G/H$ since $\cl{\Gamma C}$ is $H$-invariant and $g^{*}\in\cl{\Gamma C}$, which we consider as a subset of $G$.
\end{proof}

\section{Closed and dense orbits}
	In this chapter, we will investigate the difference between closed and dense $\Gamma$-orbit of a circle $C$. 
We will consider the closed orbit case first.
The properties of closed orbits heavily rely on the combinatorial nature of the Bruhat-Tits tree.
We will use results from Chapter 3 to prove the properties of dense orbits.
The following lemma is simple yet important.
The proof is as in \cite{MMO}.

\begin{lem}\label{lem:closed_discrete}
Let $\Gamma$ be a finitely generated subgroup of $G$. If $\Gamma C$ is closed, then it is discrete in $\cal{C}=G/H$.
\end{lem}

\subsection{Closed orbits in $\cal{C}_\Lambda$}

For a circle $C=g_CH$, recall that $\hull(C)=g_C.T_H$.
Recall also that  
$
	\Gamma^{C}=\op{Stab}_{\Gamma}(C)=g_{C}(\Gamma\cap H)g_{C}^{-1}.
$
For every vertex $\bv$ in $g_C.T_H$, we define a subset $\Gamma(C,\bv)$ of $\Gamma$ by
$$
    \Gamma(C,\bv):=\{\gamma\in\Gamma:\bv\in \gamma g_C.T_H\}/\sim,
$$
where 
$
	\gamma\sim\tau 
	\hspace{3pt}
	\Longleftrightarrow 
	\hspace{3pt}
	\gamma g_{C}.T_{H}=\tau g_{C}.T_{H}.
$

For a subset $A$ of $T_{G}$, 
we say $A$ is \emph{locally $\Gamma C$-finite} if $|\Gamma(C,\bv)|<\infty$ for every $\bv\in A$.
Recall also that $\Lambda^{*}=\Lambda-\Lambda_{\ax}$, where $\Lambda_{\ax}=\{\gamma^{\pm}:\gamma\in\Gamma\}$.
Let 
\begin{equation}\label{lambdaC}
\Lambda_{\ax}^{C}=\{\gamma^{\pm}:\gamma\in\Gamma^{C}\}.
\end{equation}
%If there exists an element $\gamma\in\Gamma$ such that $\chi(\gamma)=\al\omg$ with $\al\in\zp^{\xx}$, define a set 
%$$
%	\Lambda_{\omg}=
%	\{\gamma^{\pm}:\chi(\gamma)=\al\omg,\al\in\zp^{\xx}
%	\}.
%$$

\begin{lemma}\label{lem:fin_lim_circles}
Let $C=g_CH$ for some $g_C\in G$. Suppose that $\Gamma C$ is closed. 
Then,
\begin{enumerate}
	\item For every vertex $\bv\in g_C. T_H$, $\Gamma(C,\bv)$ is finite.
	\item $C\cap\Lambda\subseteq\Lambda^{*}\cup\Lambda_{\ax}^{C}$.
\end{enumerate}
\end{lemma}

\begin{proof}(1)
 Suppose that for a vertex $\bv\in g_{C}.T_{H}$, there is an infinite sequence of distinct elements 
 $
    \{[\rho_i]\}_{i=1}^\infty\subset\Gamma(C,\bv).
 $
For every $n\in\z_{\geq1}$, one can construct a subsequence $\{\rho_{i_{n}}g_{C}.T_{H}\}_{j=0}^{\infty}$ of $H$-subtrees such that,
 $$
 	N_{n}(\bv)\cap \rho_{i_{n}}g_{C}.T_{H}=N_{n}(\bv)\cap(\cap_{j=n}^{\infty}\rho_{i_j}g_{C}.T_{H}).
 $$
Since $\Gamma C$ is closed, $\rho_{i_{n}}g_{C}.T_{H}$ converges in $\Gamma \hull(C)$ as $n\rightarrow \infty$.
Since $\Gamma C$ is discrete, this is a contradiction.

\noindent(2) 
Suppose that there exists $a\in C\cap(\Lambda_{\ax}-\Lambda_{\ax}^{C})$.
Since $a$ is an element of $\Lambda_{\ax}$, there exists $\gamma\in\Gamma$ such that $\gamma^{-}=a$.
By definition (\ref{lambdaC}), 
it follows that $\gamma\in\Gamma-\Gamma^{C}$ since $\gamma^{-}\in\Lambda_{\ax}-\Lambda_{\ax}^{C}$.
Since $\gamma^{-}\in C=\partial (g_{C}.T_{H})$,
there exists a geodesic ray $\ell$ in $g_{C}.T_{H}\cap\ax(\gamma)$.
Pick a vertex $\bv\in\ell$.
Since $\ell\subset\ax(\gamma)$, we have $\ell\subset\gamma^{n}.\ell$ so that $\ell\subset\gamma^{n}g_{C}.T_{H}$ for all $n\in\z_{\geq0}$.
It follows that $\bv\in\gamma^{n}g_{C}.T_{H}$.
Hence, $\{[\gamma^{n}]\}_{n=1}^{\infty}\subseteq\Gamma(C,\bv)$.

%Then, there exists 
%$\gamma\in\Gamma-\Gamma^{C}$
%such that $g_{C}.T_{H}\cap\ax(\gamma)\neq\varnothing$,
%in particular, $\gamma^{\pm}\notin\Lambda_{\ax}^{C}$.
%%{\color{red}
%%and $\gamma$ satisfies condition 2 of Definition \ref{def:acyl}.
%%If not, i.e., such $\gamma$ is an element of $\Gamma^{C}$, we can show a contradiction:
%%If we assume $\gamma\in\Gamma^{C}$, we have $\gamma\in g_{C}Hg_{C}^{-1}$.
%%Moreover, $(\gamma^{-},\gamma^{+})=(g_{C}.\al,g_{C}.\beta)$ for some $\al,\beta\in\qp$.
%%Thus $\chi(\gamma)\in\qp$.
%%This is a contradiction since $\chi(\gamma)\notin\qp$ by condition 2 of Definition \ref{def:acyl}.
%%}
%Taking $\gamma^{-1}$ if necessary, we may assume that $\gamma^{-}\in C\cap\Lambda$ and pick some $\bv\in g_{C}.T_{H}\cap\ax(\gamma)$.
%Since $\ax(\gamma)$ is invariant under $\gamma$,
%we obtain $\{[\gamma^{n}]\}_{n=1}^{\infty}\subseteq\Gamma(C,\bv)$.

If $\{[\gamma^{n}]\}_{n=1}^{\infty}$ is finite, then there exists $m\in\z_{>1}$ such that $\gamma^{m}\in\Gamma^{C}$.
Hence $\gamma^{\pm}\in\Lambda_{\ax}^{C}$,
which is a contradiction since $\gamma^{-}=a\notin\Lambda_{\ax}^{C}$.

If $\{[\gamma^{n}]\}_{n=1}^{\infty}$ is infinite, then $\Gamma(C,\bv)$ is infinite, which is a contradiction by part (1).

\end{proof}

As a corollary of Lemma \ref{lem:fin_lim_circles},
when $\Gamma C$ is closed,
every $H$-subtree defined by an element of $\Gamma C$ is locally $\Gamma C$-finite.

\begin{thm}\label{thm:discreteorbit}
Let $C\in\cal{C}$. 
If $\Gamma C$ is closed and $C\cap\Lambda\neq\varnothing$, then $\Gamma^C$ is finitely generated and $C\cap\Lambda$ is the limit set of $\Gamma^C=\op{Stab}_{\Gamma}(C)$.
\end{thm}

\begin{proof}
Let $C=g_{C}H$ for some $g_{C}\in G$.
	Let 
$
	\pi_{C}:T_{G}\rightarrow\Gamma^{C}\backslash T_{G}
$
be the quotient map and let $X_{C}=\Gamma^{C}\backslash g_{C}.T_{H}$.
Since $\Gamma C$ is closed, by Lemma \ref{lem:fin_lim_circles}, any vertex in $\Gamma g_{C}.T_{H}$ is locally $\Gamma C$-finite.
Thus for each vertex in the subset $\pi_{C}^{-1}(X_{C})$ of $T_{G}$ is also locally $\Gamma C$-finite.

Define
$
	f:X_{C}\rightarrow X=\Gamma\backslash T_{G}
$
by $f(\Gamma^{C}.\bv)=\Gamma.\bv$.
Note that
$$f^{-1}(\Gamma.\bv)=\{\Gamma^{C}.\bw:\bw\in\Gamma.{\bv}\cap g_{C}.T_{H}\}.$$
We first claim that $f$ is proper.
Indeed, 
if $f^{-1}(\Gamma.\bv)$ is infinite,
there exists an infinite sequence $\Gamma^{C}\gamma_{i}.\bv\in f^{-1}(\Gamma.\bv)$.
It follows that
$$\bv=\gamma_{i}^{-1}.\bw\in\gamma_{i}^{-1}.(\Gamma.\bv \cap g_{C}.T_{H})\subset \gamma_{i}^{-1}g_{C}.T_{H},$$
thus $[\gamma_{i}^{-1}]\in\Gamma(C,\bv)$.
Since $\Gamma(C,\bv)$ is finite by locally $\Gamma C$-finiteness,
there exists an infinite set $\{\Gamma^{C}\gamma_{i_{k}}.\bv\}$ such that $[\gamma_{i_{k}}^{-1}]$'s are the same elements in $\Gamma(C,\bv)$.
It follows that $\gamma_{i_{j}}\gamma_{i_{k}}^{-1}\in\Gamma^{C}$, but $\Gamma^{C}\gamma_{j}.\bv\neq\Gamma^{C}\gamma_{i_{k}}.\bv$, which is a contradiction.

Now define
$$
	X_{C}'=f^{-1}(\core(X)\cap f(X_{C}))=\Gamma^{C}\backslash (S_{\Gamma}\cap g_{C}.T_{H}).
$$
Note that $X_{C}'$ is a finite subset of $X_{C}$ since $\core(X)$ is finite and $f$ is proper.
Moreover, $X_{C}'=\core(X_{C})$ since
$\core(X_{C})=\Gamma^{C}\backslash(\hull(\Lambda_{C}))=\Gamma^{C}\backslash (\hull(C)\cap S_{\Gamma})$, where $\Lambda_{C}=\cl{\Lambda_{\ax}^{C}}.$
Note that the fundamental group $\pi_{1}(X_{C})$ is isomorphic to $\Gamma^{C}$.
Since $\core(X_{C})$ determines $\pi_{1}(X_{C})$, 
it follows that $\Gamma^{C}$ is finitely generated.
Moreover, we also have $\partial (\pi_{C}^{-1}(X_{C}'))=\partial(\pi_{C}^{-1}(\core(X_{C})))=\Lambda_{C}$.
Therefore, it follows that $\Lambda_{C}=\Lambda(\Gamma)\cap C$.
\end{proof}

\subsection{Dense orbits in $\cal{C}_\Lambda$}\label{ss:dense}
Recall that
$$\scc=\{e^{i\omg\theta}\in 1+p\zpo: \theta\in p\zp\}.$$
Define a function
$
\chi:\Gamma\rightarrow\zpo^\times$ by $\chi(\tau)=\Theta_\ut$, where $a_{\ut}=r_{\ut}\Theta_{\ut}$ is defined in (\ref{asubtau}).
By (2) of Definition \ref{def:acyl}, there is a hyperbolic element $\tau\in\Gamma$ such that 
$\{\chi(\tau)^{n}:n\in\z_{\geq0}\}$ is dense in 
$\mup\xx\scc$.
Let us fix such element $\ut$ throughout this subsection.
From the introduction of Section \ref{s:unip}, recall that
$$
\cal{C}_\Lambda=\{C\in\cal{C}:C\cap\Lambda\neq\varnothing\}.
$$

\begin{prop}\label{prop:denseorbit}
Let $C\in\cal{C}_{\Lambda}$ with $C=g_CH$ for some $g_C\in G$. 
If $\Gamma C$ is not closed, then $\Gamma C$ is dense in $\cal{C}_\Lambda$.
\end{prop}

\begin{proof}
Recall that $g\in G$ with $(g.0,g.\infty,g.1)$, where $g.0,g.\infty,g.1\in\cl{\qpo}$ and mutually distinct. 

	Let $C=g_{C}H$ such that $\Gamma C$ is not closed. 
Let $D=g_{D}H$ be an arbitrary element in $\cal{C}_{\Lambda}$.
Let $
		\ut=g_{\ut}d_{\ut}g_{\ut}^{-1}\sim
		d_{\ut}=\diag(p^{n_{\ut}}a_{\ut},1),
$
where $n_{\ut}\in\z_{\geq1}$ and $a_{\ut}\in\zpo^{\xx}$.

\noindent{\bf Case 1.} $\Gamma D$ is not closed:
By Proposition \ref{prop:goodcircle}, there exist $E\in\cl{\Gamma C}$ and $E'\in\cl{\Gamma D}$ such that $\ut^{\pm}\in E\cap E'$.
By Remark \ref{rmk4}, choosing suitable representatives for each circle, respectively, we have
	\begin{align}
		&(g_{E}.0,g_{E}.\infty,g_{E}.1)
		=(\ut^{-},\ut^{+},g_{E}.1)
		=(g_{\ut}.0,g_{\ut}.\infty,g_{\ut}.\al)
		\label{eqn:circle1}\\ 
		&(g_{E'}.0,g_{E'}.\infty,g_{E'}.1)
		=(\ut^{-},\ut^{+},g_{E'}.1)
		=(g_{\ut}.0,g_{\ut}.\infty,g_{\ut}.\beta)\label{eqn:circle2},
	\end{align}
where $\al=g_{\ut}^{-1}g_{E}.1$ and $\beta=g_{\ut}^{-1}g_{E'}.1$.
Let
$\al=\Theta_{\al}r_{\al}$ and $\beta=\Theta_{\beta}r_{\beta}$, where $\Theta_{\al},\Theta_{\beta}\in\mup\xx\scc$, and $r_{\al},r_{\beta}\in\qp^{(1)}$.

Letting $d_{\al}=\op{diag}(\al,1)$ and $d_{\beta}=\op{diag}(\beta,1)$, 
the equations (\ref{eqn:circle1}) and (\ref{eqn:circle2}) give us $g_{\ut}d_{\al}=g_{E}$ and $g_{\ut}d_{\beta}=g_{E'}$ 
since there is a unique element of $G$ that corresponds to each frame.

Note that
	$$
		\ut^{k}.(g_{\ut}.0,g_{\ut}.\infty,g_{\ut}.\al)
		=\ut^{k}g_{\ut}d_{\al}.(0,\infty,1)
		=g_{\ut}d^{k}_{\ut}d_{\al}.(0,\infty,1)
	$$
so that 	
	$
		\ut^{k}E=\ut^{k}(g_{E}H)=(g_{\ut}d_{\ut}^{k}d_{\al})H
	$
for every $k\in\z$.
Since $\langle a_{\ut}\rangle$ is dense in $\mup\xx\scc$, 
there is a sequence $\{m_{j}\}_{j=0}^{\infty}\subset\z_{\geq0}$ such that 
$$
	a^{m_{j}}_{\ut}\al=a^{m_{j}}_{\ut}\Theta_{\al}r_{\al}\rightarrow \Theta_{\beta}r_{\al}\in\qpo^{\xx}
$$ 
as $j\rightarrow\infty$.
For convenience, let $m_{0}=0$.

Now let us define $d=\op{diag}(p,1)\in H$.
Note that
	\begin{equation}\label{diagonal}
	{\small
		d^{m_{j}}_{\ut}d_{\al}d^{-n_{\ut}m_{j}}
		=\diag(a_{\ut}^{m_{j}}\al,1)
		}
	\end{equation}

Let $d_{r_{\al}^{-1}r_{\beta}}=\diag(r_{\al}^{-1}r_{\beta},1)\in H$.
For every $j\in\z$,
choose a representative of $\ut^{m_{j}}E
=\ut^{m_{j}}g_{E}H$
as a right $H$-coset by
	$$
		\ut^{m_{j}}g_{E}
		(
		d^{-n_{\ut}m_{j}}
		d_{r_{\al}^{-1}r_{\beta}}
		H)
		=g_{\ut}
		\diag(a_{\ut}^{m_{j}}\Theta_{\al}r_{\beta},1)H
	$$
as we observed at (\ref{diagonal}).
Denote 
$\ut^{m_{j}}g_{E}
d^{-n_{\ut}m_{j}}
d_{r_{\al}^{-1}r_{\beta}}
$
by $g_{j}\in G$.
Define a sequence of frames,
$
	\{(\ut^{-},\ut^{+},g_{\ut}.(a_{\ut}^{m_{j}}\Theta_{\al}r_{\beta}))\}_{j=0}^{\infty}.
$
Letting $j\rightarrow\infty$, we have 
$$
g_{j}\sim(\ut^{-},\ut^{+},g_{\ut}.(a_{\ut}^{m_{j}}\Theta_{\al}r_{\beta}))
\rightarrow (\ut^{-},\ut^{+},g_{\ut}.\beta)\sim g_{E'}.
$$
Hence it follows that $g_{j}\rightarrow g_{E'}$ as $j\rightarrow\infty$.
Meanwhile, for every $j$,
$$
	g_{j}H=\ut^{m_{j}}g_{E}
	(
	d^{-n_{\ut}m_{j}}
	d_{r_{\al}^{-1}r_{\beta}}
	H
	)=\ut^{m_{j}}g_{E}H=\ut^{m_{j}}E.
$$
Therefore, we conclude that
$
	\ut^{m_{j}}E\longrightarrow E'\in\cl{\Gamma C}
$
as $j\rightarrow\infty$.
	
%In particular, $\ut^{m_{j}}(g_{E}H)=(\ut^{-},\ut^{+},g_{\ut}.(a^{m_{j}}_{\ut}\al)).H$ {\co under the identifiication of an element of $G$ with the corresponding frame.}
%Thus, letting $j\rightarrow\infty$, we have
%	\begin{align*}
%		\ut^{m_{j}}(g_{E}H)
%		&{\co =\ut^{m_{j}}(g_{E}d^{-n_{\ut}m_{j}}H)
%		=g_{\ut} (d_{\ut}^{m_{j}}d_{\al}d^{-n_{\ut}m_{j}}) H  }\\
%		&=(\ut^{-},\ut^{+},g_{\ut}.(a_{\ut}^{m_{j}}\al)).H
%		\longrightarrow (\ut^{-},\ut^{+},g_{\ut}.\beta).H=g_{\ut}d_{\beta}H,
%	\end{align*}
%i.e.,
%	$
%		d_{\ut}^{m_{j}}d_{\al}d^{-n_{\ut}m_{j}}\rightarrow
%		d_{\beta}
%	$
%and 
%	$
%		\ut^{m_{j}}E\rightarrow E'\in\cl{\Gamma D}
%	$
%as $j\rightarrow\infty$.

To approach $D$ from $C$, let $\{\zeta_{i}\}_{i=0}^{\infty}\subset\Gamma$ such that $\zeta_{i}D\rightarrow E'$ in $\cl{\Gamma D}$.
There exist sequences $e_{i}\rightarrow e$ in $G-H$ 
and $\{h_{i}\}_{i=0}^{\infty}\subset H$
such that
$
	\zeta_{i}g_{D}h_{i}=g_{E'}e_{i}.
$
so that 
$
	g_{D}h_{i}e_{i}^{-1}=\zeta_{i}^{-1}g_{E'}.
$
Therefore, $D\in\cl{\Gamma C}$ since
$$
	\cl{\Gamma C} \ni\zeta_{i}^{-1}E'=\zeta_{i}^{-1}g_{E'}H
	=g_{D}h_{i}e_{i}^{-1}H
	\longrightarrow
	g_{D}H=D.
$$

\noindent{\bf Case 2.} $\Gamma D=\cl{\Gamma D}$:
Recall that $\Lambda^{*}=\Lambda-\Lambda_{\ax}$.
We know that $\Gamma D$ is discrete since it is closed.
Moreover, $D\cap\Lambda\subset\Lambda^{*}\cup\Lambda_{\ax}^{D}$ by (2) of Lemma \ref{lem:fin_lim_circles}.
Let $(a,b)\subset g_{D}.T_{H}\cap S_{\Gamma}$ be a geodesic for some distinct $a,b\in D\cap\Lambda$.
By Proposition \ref{prop:axisapp}, 
there exists a sequence $\{\zeta_{i}\}_{i=0}^{\infty}\subset \Gamma$ such that $\ax(\zeta_{i})\rightarrow(a,b)$ as $i\rightarrow\infty$.
Besides, from Section \ref{ss:ax}, we can choose $\{\zeta_{i}\}_{i=0}^{\infty}\subset\Gamma-\Gamma^{D}$ by
constructing $\zeta_{i}$ with $\ax(\zeta_{i})\not\subset \hull(D)$ for each $i$.
By choosing a subsequence if necessary, we may assume that
\begin{equation}\label{axisinc}
\ax(\zeta_{i})\cap(a,b)\subseteq\ax(\zeta_{i+1})\cap(a,b)
\end{equation} 
for all $i\in\N$.
In addition, let $f_{i}\in G$ be an element that conjugates $\zeta_{i}$ to a diagonal element.
It follows that $\zeta_{i}^{+}=f_{i}.\infty$ and $\zeta_{i}^{-}=f_{i}.0$.
Let $\bv\in\ax(\zeta_{0})\cap(a,b)$ so that $\bv\in\ax(\zeta_{i})\cap(a,b)$ for all $i$ by (\ref{axisinc}).
Now choose a vertex $\bw\in N_{1}(\bv)\cap (g_{D}.T_{H}-\ax(\zeta_{0}))$ such that $\bw\neq\bv$.
If $\bw\in (0,\infty)$, then replace $\bw$ by a vertex in $N_{1}(\bw)\cap g_{D}.T_{H}-(0,\infty).$
Let $B$ be the branch containing $\bw$ at $\bv$.
Choose $\bw_{i}\in B\cap (g_{D}.T_{H}-f_{i}.T_{H})$.
Now choose a geodesic ray $\ell_{i}=[\bw_{i},c_{i})$ in $g_{D}.T_{H}-f_{i}.T_{H}$ with some $c_{i}\in D-f_{i}.\cl{\qp}$.
Here, note that $[\bv,c_{i})\subset g_{D}.T_{H}$.

From our choice of $c_{i}$, we remark the following results.
There exists $n\in\z$ such that $c_{i}\in p^{n}\zpo$ for all $i\in\z_{\geq0}$ since $\bw\notin (0,\infty)$.
Thus, since $p^{n}\zpo$ is compact and $D=g_{D}.\cl{\qp}$ is closed in $\cl{\qpo}$, there exists a subsequence $\{c_{m_{j}}\}_{j=0}^{\infty}$ such that $c_{m_{j}}\rightarrow c$ for some $c\in D\cap p^{n}\zpo$.
By definition of $c_{i}$'s, it follows that $a,b,$ and $c$ are mutually distinct.
If we let $g_{i}$ be the element in $G$ corresponding to the frame $(\zeta_{i}^{-},\zeta_{i}^{+},c_{i})$ and $C_{i}=g_{i}H$,
then it follows that $\zeta_{i}\notin\Gamma^{C_{i}}$ since $c_{i}\notin f_{i}.\cl{\qp}$.

%In addition, let $\bv\in\ax(\zeta_{0})\cap(a,b)$ so that $\bv\in\ax(\zeta_{i})\cap(a,b)$ for all $i$ by (\ref{axisinc}).
%Now there exists a vertex $\bw_{i}\in N_{1}(\bv)\cap (g_{D}.T_{H}-\ax(\zeta_{0}))$ such that $\bw_{i}\neq\bv$.
%Furthermore, we can choose a geodesic $\ell_{i}\subset g_{D}.T_{H}$ such that $\bv\notin\ell_{i}$ and $\bw_{i}\in\ell_{i}$.
%Let $c_{i}\in g_{D}.\cl{\qp}$ be an end point of $\ell$.
%Note that $\ell\cap \ax(\zeta_{i})=\varnothing$ for all $i\in\z_{\geq0}$.
%Hence $\zeta_{i}^{-},\zeta_{i}^{+}$, and $c$ are mutually distinct elements of $\cl{\qpo}$ for all $i$.

%Define a sequence of frames 
%$
%	\{(\zeta_{i}^{-},\zeta_{i}^{+},c)\}_{i=0}^{\infty}.
%$
%Let $g_{i}\in G$ be an element corresponding to the frame $(\zeta_{i}^{-},\zeta_{i}^{+},c)$,
%i.e., $(\zeta_{i}^{-},\zeta_{i}^{+},c)=(g_{i}.0,g_{i}.\infty,g_{i}.1)$.
Finally, for any $i\in\z_{\geq0}$, $C_{i}$ has the $\Gamma$-orbit that is not closed since $\hull(C_{i})$ contains $\ax(\zeta_{i})$ while $\zeta_{i}\notin\Gamma^{C_{i}}$, thus by (2) of Lemma \ref{lem:fin_lim_circles}.
Hence, by Case 1, there exists a sequence $\{\rho_{i,l}\}_{l=0}^{\infty}\subset\Gamma$ 
such that $\rho_{i,l}C\rightarrow C_{i}$ as $l\rightarrow\infty$ for every $i$.
Consequently, we obtain a converging sequence of circles $\{C_{m_{j}}\}_{j=0}^{\infty}\subset\cl{\Gamma C}$, which converges to $D$.
\end{proof}

\subsection{Proof of main theorem}
Finally, we proceed to the proof of Theorem \ref{thm:main} through the following proposition.
%\subsection{Proof of Theorem \ref{thm:main}.}
\begin{prop}
Let $I=C\cap\Lambda$ be the intersection of a circle $C$ with the limit set $\Lambda$ of $\Gamma$. Then either $I=\varnothing$,  $I=C$, or $\abs{I}=\infty$ and $I\notin\Lambda$.
\end{prop}

\begin{proof}[proof of Theorem \ref{thm:main}]
Let $C=g_{C}H$ for some $g_{C}\in G$.
If $I=\varnothing$, then $\Gamma C$ is discrete since $\Gamma$ acts properly discontinuously on $\partial T_G-\Lambda$ and $\Gamma\backslash g_C.T_H$ is an infinite subtree of $\Gamma\backslash T_G$.

If $I\neq\varnothing$, then we will consider the following two cases. First, if $\Gamma C$ is closed, then it is discrete by Lemma \ref{lem:closed_discrete}. By Theorem \ref{thm:discreteorbit}, $I=\Lambda(\Gamma^C)$ and $\Gamma^C$ is conjugate to a finitely generated Schottky subgroup of $\PGLt(\qp)$ since $\Gamma^C=gHg^{-1}\cap\Gamma$. If $\Gamma C$ is not closed, then, by Proposition \ref{prop:denseorbit}, we have $\cl{\Gamma C}=\cal{C}_\Lambda$.

\end{proof}

\end{document}